\documentclass[12pt, reqno]{amsart}
\input{deflea.sty}
\usepackage[utf8]{inputenc}
\usepackage{verbatim}
\usepackage{cases}
\usepackage{booktabs}

\counterwithout{table}{section}

\usepackage[justification=centering]{caption}

\title[$p$-adically convergent loci for periodic continued fractions]{$p$-adically convergent loci in varieties arising from periodic continued fractions}

\author{Laura Capuano}
\address[Laura Capuano]{Dipartimento di Matematica e Fisica, Università degli Studi Roma Tre}
\email{laura.capuano@uniroma3.it}
\author{Marzio Mula}
\address[Marzio Mula]{Research Institute CODE, University of the Bundeswehr Munich}
\email{marzio.mula@unibw.de}
\author{Lea Terracini}
\address[Lea Terracini]{Dipartimento di Informatica, Università degli Studi di Torino}
\email{lea.terracini@unito.it}
\author{Francesco Veneziano}
\address[Francesco Veneziano]{Dipartimento di Matematica, Universit\`a degli Studi di Genova}
\email{francesco.veneziano@unige.it}

\date{\today}

\begin{document}

\begin{abstract}
Inspired by the existence of multiple alternative definitions of continued fraction expansions for elements in $\QQ_p$, we study the $p$-adic convergence of periodic continued fractions with partial quotients in $\ZZ[1/p]$ from a geometric point of view. To this end, following a previous work by Brock, Elkies, and Jordan, we consider certain algebraic varieties whose points represent formal periodic continued fractions with period and preperiod of fixed lengths, satisfying a given quadratic equation. We then focus on the \emph{$p$-adically convergent loci} of these varieties, describing the zero and one-dimensional cases by combining tools from algebraic geometry, arithmetic, and the theory of Pell equations and of linear recurrences.
\end{abstract}

\maketitle

\section{Introduction}
Given an element $\alpha\in \mathbb{R}$, the classical \emph{(Archimedean) continued fraction expansion} of $\alpha$ is  computed by means of the following algorithm: \begin{equation}
    \label{eqn:CFalg}
    \begin{cases} \alpha_0 =\alpha\\
c_n =\lfloor\alpha_n\rfloor \\
    \alpha_{n+1} = \frac 1 {\alpha_n-c_n} \quad \textrm{if } \alpha_n-c_n \neq 0,\end{cases} 
\end{equation}
stopping if $\alpha_n=c_n$, where $\lfloor \cdot \rfloor$ denotes the integral floor function. It is well known that $\alpha$ is rational if and only if its continued fraction expansion is finite, and, if $\alpha$ is not rational, then the sequence of \emph{$n$-th convergents}
\begin{align*}
    [c_1,\dotsc,c_n]=c_1 + \cfrac{1}{c_2+ \cfrac{1}{\ddots + \cfrac{1}{c_n}}}
\end{align*}
converges to $\alpha$. In the latter case, eventually periodic sequences (PCFs) are of particular interest: their study dates back to Lagrange~\cite{lagrange1770}, who proved that the continued fraction expansion of $\alpha$ is eventually periodic if and only if $\alpha$ is a quadratic irrational.
\par In~\cite{BrockElkies2021} the authors investigate periodic continued fractions using the language and techniques of algebraic geometry. More specifically, given a multi-set $\{\beta,\beta^\ast\}\subseteq \mathbb{P}^1(\overline{\mathbb Q})$, and two positive integers $N$ and $k$, the authors define a variety $V_{N,k}(\beta,\beta^\ast)\subseteq {\mathbb A}^{N+k}$, called \emph{PCF variety}, whose points $(b_1,\dotsc,b_N, a_1,\dotsc,a_k)$, seen as the continued fraction $[b_1,\dotsc,b_N,\overline{ a_1,\dotsc,a_k}]$ with pre-period of length $N$ and period of length $k$, formally satisfy the quadratic equation having $\beta$ and $\beta^\ast$ as roots. They particularly focus on PCF varieties of dimension $\leq 1$ and consider points in some ring of $S$-integers $\calo_S$ in a number field $K$. We stress that these sequences need not arise as continued fraction expansions nor be convergent. For any ring $A\subseteq \CC$, the set $V_{N,k}(\beta,\beta^*)(A)$ of $A$-rational points on $V_{N,k}(\beta,\beta^*)$ has a notable subset: the \emph{convergent locus}, whose points correspond to continued fractions that converge in $\CC$. In \cite[\S 4]{BrockElkies2021} the authors also give an algorithm to detect the convergence of a given continued fraction.
\par We lay out a $p$-adic analogue of \cite{BrockElkies2021}, where $p$ is an odd prime. Various constructions of continued fractions can be carried out in the $p$-adic setting: the classical ones date back to Mahler~\cite{Mahler1940}, Ruban~\cite{Ruban1970}, Browkin~\cite{Browkin1978} and Schneider~\cite{Schneider1970}.  More recently, periodic $p$-adic continued fractions have been investigated by~\cites{Bedocchi1990}{Ooto02017}{CapuanoVenezianoZannier2019}{CapuanoMurruTerracini2023}{MurruRomeo2023}{Romeo2024}. 
\par In this work, 
we consider the $p$-adic convergence of \emph{any} eventually periodic continued fraction $[c_1,c_2,\dots ]$ with partial quotients $c_i\in \mathbb{Z}[1/p]$. We remark that our approach does not consider a fixed algorithm such as Ruban's or Browkin's to generate an expansion of any given number, but we consider every eventually periodic sequence of partial quotients. 
To this aim, we keep the same definition of PCF varieties as in~\cite{BrockElkies2021} and study their \emph{$p$-adically convergent locus}, i.e.\ the set of points corresponding to continued fractions which converge w.r.t.\ the $p$-adic---rather than the Archimedean---absolute value.

The main tool to study the $p$-adically convergent loci of PCF varieties is Theorem~\ref{criterio.di.convergenza.padica}, which provides a criterion for the $p$-adic convergence of a periodic continued fraction; the criterion requires checking finitely many $p$-adic inequalities involving the partial quotients in the periodic part of the continued fraction. 
\par Our notation slightly differs from that of~\cite{BrockElkies2021}: rather than labelling the PCF varieties by the multiset $\mathcal{B}=\{\beta,\beta^*\}$ consisting of the two roots of some polynomial $F= Ax^2+Bx+C\in\QQ[x]$, we will instead use $F$ itself and write $V(F)_{N,k}$. 
This paper aims to describe the $p$-adic convergent locus of $V(F)_{N,k}$ in terms of the coefficients of $F$, for PCF varieties of small dimension. This is usually done in two steps: we first characterize the $\ZZ[1/p]$-rational points of $V(F)_{N,k}$, and then use Theorem~\ref{criterio.di.convergenza.padica} to see which of these points converge. As a result, for $(N,k)\in \{(0,1),(1,1),(0,2)\}$, we manage to explicitly enumerate the $p$-adically convergent points of $V(F)_{N,k}$ for any given $F$, expressing them as functions of the coefficients of $F$ (Propositions~\ref{prop:PCF01}, \ref{prop:PCF11}, and~\ref{prop:PCF02}). In the remaining cases, the relation between $p$-adically convergent points and $F$ becomes less transparent. For $V(F)_{2,1}$ and $V(F)_{0,3}$, we only prove the finiteness of the $p$-adic locus (Propositions~\ref{prop:21a}, \ref{prop:21b}, and \ref{prop:soluzioni0-3}), while in Section~\ref {subsec:12} we obtain a geometric description of  $V(F)_{1,2}(\ZZ[1/p])$ and, in Theorem~\ref{teo:cond1-2}, we specialize Theorem~\ref{criterio.di.convergenza.padica} to this case. Our results are summarized in Table~\ref{tab:summary}.

\begingroup
\setlength{\tabcolsep}{4pt}
\renewcommand{\arraystretch}{2} 
\setlength{\aboverulesep}{1.0ex}   
\setlength{\belowrulesep}{0.8ex}   
\begin{table}[t]
\centering
\small
\begin{tabular}{@{}lcc@{}}
\toprule
$(N,k)$ & $\#V(\ZZ[1/p])^{\mathrm{con}}$ & Main result \\
\midrule
\hyperref[sect:01]{(0,1)} & $\le 1$ & Prop.~\ref{prop:PCF01} \\
\hyperref[sect:11]{(1,1)} & $0\ \text{or}\ 2$ & Prop.~\ref{prop:PCF11} \\
\hyperref[sect:21]{(2,1)} &
$\begin{cases}
\text{$<\infty$} & \text{if } A\neq 0,\\
0\ \text{or}\ 2\  & \text{if } A=0
\end{cases}$ &
Props.~\ref{prop:21a}, \ref{prop:21b}\\
\hyperref[sect:02]{(0,2)} & $\le 1$ & Prop.~\ref{prop:PCF02} \\
\hyperref[sect:12]{(1,2)} &
$<\infty$ &
Cor.\,\ref{cor:Vcon12} \\
\hyperref[sect:03]{(0,3)} &
$<\infty$ &
Prop.~\ref{prop:soluzioni0-3} \\
\bottomrule
\end{tabular}
\caption{Some of our results on the number of $p$-adically convergent points in $V(F)_{N,k}$, where $F=Ax^2+Bx+C$.}
\label{tab:summary}
\end{table}
\endgroup
  For the type $(0,3)$, in Section \ref{sssect:pureintegerradicals} we also focus on the ``pure radical'' case, i.e.\ when $F$ is a quadratic polynomial of the form $x^2-d$, with $d$ an integer. As shown in Remark~\ref{rem:empty03}, the relevant case is when $d>1$. 
 In this setting, the defining equations of the PCF give rise to generalized Pell equations, so that the existence and structure of convergent periodic $p$-adic expansions for $\sqrt d$ are ruled by solutions of Pell-type diophantine equations (see Proposition~\ref{prop:soluzioni0-3}).
The well-known theory of generalized Pell equations describes how families of solutions can be recursively constructed, via the unit group action, from a finite set of fundamental solutions. We use this theory as a tool that
(i) yields effective finiteness results by combining Pell recurrences with arithmetic constraints on perfect powers in linear recurrence sequences (Section~\ref{sec:finiteness03}) 
and (ii) allows one to construct explicit families of $p$-adically convergent periodic continued fractions for $\sqrt d$ (Sections~\ref{sec:a2+1} and~\ref{sss:4.6.4}).

 Finally, in the case  $(1,3)$, we obtain explicit families of $p$-adically convergent periodic expansions for some quadratic irrationals. In particular, we construct convergent periodic continued fractions for $p$-adic square roots of infinitely many integers $d$, including examples of the form $\sqrt{p^2+1}$ and, more generally, of the form $\sqrt{\a^2+1}$. 

\section{Preliminaries}
\label{sec:preliminaries}
We recall some basic facts from \cite[\S2]{BrockElkies2021}.
Let $R$ be an integral domain and $K$ be its fraction field. Let $c_1,c_2,\dots$ be a sequence of elements in $R$. We define the classical sequences
\begin{align*}
    A_{0}=1,&& A_1=c_1, && A_n=A_{n-1}c_n + A_{n-2} &&\text{and}\\
    B_{0}=0, && B_1=1, && B_n=B_{n-1}c_n +B_{n-2} &&\text{for $n \geq 2$}.
    \end {align*}
     A \emph{finite continued fraction} is an element   of the form $[A_n : B_n]\in \mathbb{P}^1(K)$, which we shall denote by $  [c_1,c_2,\dotsc, c_n]$. This is always a well-defined point of  $\mathbb{P}^1(K)$, due to the classical formula $A_{n+1}B_n-A_{n}B_{n+1}=(-1)^{n+1}$.\\
    When $B_n \neq 0$, we identify this point with the value \[ [c_1,c_2,\dotsc, c_n]=c_1 + \cfrac{1}{c_2+ \cfrac{1}{\ddots + \cfrac{1}{c_n}}}=\frac{A_n}{B_n}\in K\]
    and we write $\infty$ for the point $[1:0]$. We define a \emph{continued fraction} $\mathcal{C}=[c_1,c_2,\dots]$ as the formal expression
\begin{align}
\label{eqn:generalCF}
    c_1 + \cfrac{1}{c_2+ \cfrac{1}{\ddots}}.
\end{align}

\par When $K$ is endowed with a metric,  denoting by $\widehat K$  its completion, we shall say that $\mathcal{C}$ \emph{converges} to $\alpha\in \PP^1(\widehat K)$ if 
$$\lim_{n\to\infty}[c_1,\ldots, c_n]=\alpha$$
in the standard topology of $ \PP^1(\widehat K)$.
\\

Finally, for a continued fraction $\mathcal{C}=[c_1,c_2, \dots]$ we define the matrices
    \begin{align}
    \label{eqn:MnDn}
    D_n=D_n(\mathcal{C}) =\begin{pmatrix}
    c_n & 1 \\ 
    1 & 0
    \end{pmatrix}
   && \text{and}&&
     M_n=M_n(\mathcal{C}) =\begin{pmatrix}
    A_n & A_{n-1} \\ 
    B_n & B_{n-1}
    \end{pmatrix}
    && \text{for $n \geq 1$}.
\end{align}
If $\mathcal{C}=[c_1, \dots, c_n]$, we simply write $M(c_1, \ldots, c_n)$ instead of $M_n([c_1, \ldots, c_n])$.\\
The following properties of $D_n$ and $M_n$ are easily derived from the definitions.
\begin{proposition}
\label{prop:DnMn}
Let $\{c_n\}_{n \in \mathbb{N}}$ be any sequence of elements in $R$, and $M_n, D_n$ be defined as in~\eqref{eqn:MnDn}. Then, for each $n\geq 1$,
\begin{enumerate}[label=(\alph*)]
    \item $M_n=D_1\cdots D_{n}$,
    \item $M_n=M(c_1, \dots, c_j)M(c_{j+1},\dots, c_n)$ for each $j<n$,
    \item $M_n^{-1}=\begin{pmatrix}
        0 & 1 \\ 1 & 0
    \end{pmatrix}\cdot M(-c_1,\dots,-c_n)\cdot \begin{pmatrix}
        0 & 1 \\ 1 & 0
    \end{pmatrix}$,
    \label{DnMd:4}\item $M_n^T = M(c_n,\dotsc,c_1)$,
    \item $\det(M_n)=(-1)^{n}$.
\end{enumerate}
\end{proposition}

\subsection{PCF varieties}
\par A \emph{periodic continued fraction} (or \emph{PCF} for short) \emph{of type $(N,k)$} is a continued fraction of the form \[
\mathcal{P}=[\underbrace{b_1, \dots, b_N,}_\text{preperiod} \underbrace{a_1,..., a_k,}_\text{period}a_1,a_2,\dots],
\] 
and is usually denoted by $[b_1,\dots,  b_N,\overline{a_1,\dots,a_k}]$. When there is no preperiod, i.e.\ $N=0$, we say that $\mathcal{P}$ is \emph{purely periodic} (or \emph{PPCF} for short).

\par As shown in~\cite{BrockElkies2021}, each type $(N,k)$ and quadratic polynomial $F(x)\in K[x]$  give rise to an algebraic variety whose points represent  PCFs of type $(N,k)$, \emph{disregarding the matter of convergence}. 
\par Let $\mathcal{P}=[b_1,\dots,  b_N,\overline{a_1,\dots,a_k}]$ be a periodic continued fraction. We define the following $2 \times 2$ matrix  with coefficients in $R$:
\begin{equation}
    \label{eqn:E}
    E=\begin{pmatrix}
E_{11}(\mathcal{P})&E_{12}(\mathcal{P})\\
E_{21}(\mathcal{P})&E_{22}(\mathcal{P})
\end{pmatrix}
=M(b_1,\dots,b_N) M(a_1,\dots, a_k)M(b_1,\dots,b_N)^{-1}.
\end{equation}

To each PCF we associate the polynomial
\begin{equation}
    \label{eqn:quadf}
    \mathrm{Quad}(\mathcal{P})=E_{21}(\mathcal{P}) x^2 + (E_{22}-E_{11})(\mathcal{P})x - E_{12}(\mathcal{P})\in R[x].
\end{equation}
As in \cite{BrockElkies2021}, for a matrix $M=\begin{pmatrix} a & b\\ c & d\end{pmatrix} \in \mathrm{GL}_2(K)$, we shall denote by $\overline M$ the associated automorphism of $\PP^1$
$$\overline M(\alpha)=\frac {a\alpha+b}{c\alpha+d}.$$

We gather in Table~\ref{tab:polynomials} the polynomials $\mathrm{Quad}(\mathcal{P})$ for the types that will be considered in the remainder of this article.
   \begin{center}
\begin{table}[ht]
 \begin{tabular}{ |l|l| } 
 \hline
 $\mathcal{P}$ & $\mathrm{Quad}(\mathcal{P})$ \\ 
 \hline
 $[\overline{a_1}]$ & $x^2 -a_1 x - 1$ \\ \hline 
 $[b_1,\overline{a_1}]$ & $x^2 + (a_1 - 2b_1)x + b_1^2 - a_1 b_1 - 1$\\
 \hline
 $[b_1,b_2,\overline{a_1}]$ & \makecell{$(b_2 a_1 -b_2^2+1) x^2 + (-2a_1 b_1 b_2 +2 b_1b_2^2 -a_1-2 b_1+2 b_2)x+$\\ $a_1 b_1^2 b_2 - b_1^2 b_2^2 +a_1 b_1 + b_1^2 - 2 b_2 b_1 -1$}\\
 \hline
 $[\overline{a_1,a_2}]$ & $a_2 x^2 - a_1 a_2 x - a_1$\\ \hline
 $[b_1,\overline{a_1,a_2}]$ & $a_1 x^2 + (a_2 a_1 -2 b_1 a_1)x-a_1 a_2 b_1+a_1 b_1^2-a_2$\\ 
 \hline
  $[\overline{a_1,a_2,a_3}]$ & $(a_2 a_3 +1) x^2 + (-a_1 a_2 a_3-a_1+a_2-a_3)x -a_2 a_1 -1$\\
\hline
  $[b_1,\overline{a_1,a_2,a_3}]$ & \makecell{$(a_1a_2+1) x^2 + (a_1a_2a_3-2a_1a_2b_1+a_1-a_2+a_3-2b_1)x $ \\ $ -a_1a_2a_3b_1+a_1a_2b_1^2-a_1b_1-a_2a_3+a_2b_1-a_3b_1+b_1^2-1 $}\\
\hline
\end{tabular}
 \caption{The polynomial $\mathrm{Quad}(\mathcal{P})$ for some small values of $N,k$.}
 \label{tab:polynomials}
\end{table}
\end{center}
If $K$ is a normed field and $\mathcal{P}$ converges to $\alpha\in\PP^1(\widehat K)$, then it is clear that $\overline{E}(\alpha)=\alpha$, i.e.\ $\alpha$ must be a fixed point of the automorphism 
$\overline{E}$. This justifies the following 
\begin{proposition}
[Proposition 2.9 \cite{BrockElkies2021}] Assume that $K$ is a topological field. If $\mathcal{P}$ converges to $\alpha\in\PP^1(\widehat{K})$, then $\alpha$ is a root of the polynomial $\mathrm{Quad}(\mathcal{P})$.

\end{proposition}

 In order to define the roots of $\mathrm{Quad}(\mathcal{P})$, we view quadratic polynomials as homogeneous polynomials $AX^2 + BXY + CY^2$, so that their roots are well-defined in $\mathbb{P}^1(\overline{K})$ (except when $A=B=C=0$, which we discuss in Remark~\ref{rem:zeropol}). To ease notation, we will keep the affine notation and say that $\infty$ is a root
of $0x^2 + Bx + C$ and a double root of $0x^2 + 0x + C$. With this convention, either $\mathrm{Quad}(\mathcal{P})$ has two roots (counted with multiplicity) or it is the zero polynomial.

\par Thus we have attached a polynomial to each PCF. Now, following~\cite[§3]{BrockElkies2021}, we want to do the opposite: given a (non-zero) quadratic polynomial $F(x)=Ax^2+Bx+C\in K[x]$ and a type $(N,k)$, we consider the set of $(N+k)$-tuples $(b_1,\dots,  b_N,a_1,\dots,a_k)$ corresponding to periodic continued fractions $\mathcal{P}=[b_1,\dots,  b_N,\overline{a_1,\dots,a_k}]$ such that the polynomial $\mathrm{Quad}(\mathcal{P})$ from~\eqref{eqn:quadf} is  multiple of $F(x)$ by a non-zero scalar. They all belong to $V(F)_{N,k}(R)$, where $V(F)_{N,k}$ is the algebraic variety in $\AA^{N+k}$ defined by the equations
\begin{equation} \label{eq:VF}
\begin{cases} 
    A(E_{22} - E_{11})(\mathcal{P}) = BE_{21}(\mathcal{P}),\\
-AE_{12}(\mathcal{P}) = CE_{21}(\mathcal{P}),\\
-BE_{12}(\mathcal{P}) = C(E_{22} - E_{11})(\mathcal{P}),
\end{cases}
\end{equation}
which has generically dimension $N+k-2$ according to~\cite[§3.1 a)]{BrockElkies2021}.

\subsection{\texorpdfstring{$p$}{p}-adically convergent loci}
From now on we will focus on the $p$-adic setting, for~$p$ a fixed odd prime. We will denote by $|\cdot|_p$ the $p$-adic absolute value, and by $v_p(\cdot)$ the $p$-adic valuation.

 Having attached points of $\AA^{N+k}$ to (formal) continued fractions, we may now study their $p$-adic convergence. Inspired by the most studied continued fraction algorithms in the $p$-adic setting, from now on we will consider partial quotients in $R=\ZZ\left[1/p\right]$. To ease the notation, we will denote $\ZZ\left[1/p\right]$ by $\calo$.  Then we denote by $V(F)_{N,k}^{\mathrm{con}}$ the \emph{$p$-adically convergent locus of $V(F)_{N,k}(\calo)$}, i.e.\ the subset of $V(F)_{N,k}(\calo)$ consisting of the points  corresponding to $p$-adically convergent PCFs.
\begin{remark}
    If $F$ has no root in $\PP^1(\mathbb{Q}_p)$, then $V(F)_{N,k}^{\mathrm{con}}$ is trivially empty.
\end{remark}

\begin{remark}
\label{rem:zeropol}
    The case in which $F$ is the zero polynomial corresponds to periodic continued fractions $\mathcal{P}=[b_1,\dots,  b_N,\overline{a_1,\dots,a_k}]$ such that the matrix $E$, or, equivalently, $M(a_1,\dots, a_k)$ (note that this does not depend on the pre-period), is a multiple of the identity matrix. The corresponding variety, denoted by $V_{N,k}$, is described in~\cite[Prop.\,3.2]{BrockElkies2021}. An immediate consequence of Theorem~\ref{criterio.di.convergenza.padica} is that its convergent locus is empty.
\end{remark}

\begin{proposition}\label{prop:scambio}
Let $F=A x^2 +B x +C$ be a polynomial and $G=C x^2 -Bx +A$; let $k$ be a positive integer. The varieties $V_{0,k}(F)$ and $V_{0,k}(G)$ are isomorphic through the map
\begin{align*}
    \sigma_k: V_{0,k}(F)&\to V_{0,k}(G)\\
    (a_1,\dotsc, a_k)&\mapsto (a_k,\dotsc,a_1)
\end{align*}
\end{proposition}
This is a consequence of Proposition~\ref{prop:DnMn}\ref{DnMd:4}: transposing the matrix $E$, exchanging $A$ with $C$ and changing the sign of $B$ swaps the first and last equation of \eqref{eq:VF} and changes the sign of the middle one.

\section{Convergence criterion for PCF varieties}
\label{sec:convergence}
Now we prove a criterion for the $p$-adic convergence of a purely periodic continued fraction; this can be seen as a $p$-adic analogue of \cite[Thm.\,4.3]{BrockElkies2021}. The convergence of $p$-adic continued fractions in various settings  has already been extensively considered. Our result, concerning only periodic continued fractions, states a simple criterion which is both necessary and sufficient; this partially overlaps with the sufficient criteria given in \cite[Lemma 1]{Browkin1978}, \cite[Lemma 1]{Browkin2000}, and \cite{Murru2023}.

 From now on we will denote by $p$ a fixed odd prime. Since it is clear that the convergence of a PCF depends only on its periodic part, we can temporarily restrict our attention to a $\PPCF$ $\mathcal{P}=[\overline{a_1,\ldots, a_{k}}]$, with $a_1,\ldots, a_k\in\calo$. In this case, the matrix $E$ defined in~\eqref{eqn:E} is simply $M_k$ and, consequently, the polynomial $\mathrm{Quad}(\mathcal{P})$ from \eqref{eqn:quadf} is $B_{k}x^2+ (B_{k-1}-A_{k})x - A_{k-1}$.

\begin{theorem}\label{criterio.di.convergenza.padica}
Let $\mathcal{P}=[\overline{a_1,\ldots , a_{k}}]$ be a $\PPCF$. Then $\mathcal{P}$ is $p$-adically convergent if and only if the following  conditions are satisfied:
\begin{enumerate}[label=(\alph*)]
    \item[(i)] $|A_{k}+B_{k-1}|_p>1$;
      \item[(ii)] for $j=1,\ldots, k$, if $M(a_j,\dotsc , a_{j+k-1})_{21}=0$ then $|M(a_j,\dotsc , a_{j+k-1})_{22}|_p <1$.

    \end{enumerate}
\end{theorem}
\begin{proof}
For $j=1,\ldots , k$ we define the continued fraction $\mathcal{P}_j=[\overline{a_j,\ldots a_{k+j-1}}]$  and the matrix $T_j=M(a_j,\ldots a_{k+j-1})$.We observe that each $T_j$ is conjugate to $T_1$ because, by Proposition~\ref{prop:DnMn},
\[T_1= D_1 \cdots D_{k}=\left( D_1 \cdots D_{j-1}\right) \underbrace{\left( D_j \cdots D_{k}\right) \left( D_1 \cdots D_{j-1}\right)}_{T_j} \left( D_1 \cdots D_{j-1}\right)^{-1}. \]
Therefore, since the characteristic polynomial of a matrix is invariant under conjugation,
$$\mbox{Tr}(T_j)=A_k+B_{k-1},\quad \hbox{for } j=1,\ldots, k.$$
First, we prove that, if $\mathcal{P}$ is convergent, then $i$) holds. Suppose by contradiction that $\mathcal{P}$ is convergent and $|A_{k}+B_{k-1}|_p\leq 1$. The characteristic polynomial of the matrix
$$M_{k}=\begin{pmatrix} A_{k} &A_{k-1}\\ B_{k} & B_{k-1}\end{pmatrix}$$
has two roots $\mu,\nu$ such that $\mu\nu=(-1)^k$ and $|\mu+\nu|_p\leq 1$, therefore $|\mu|_p=|\nu|_p=1$. 
We will show that this contradicts the convergence assumption by distinguishing three cases.
\begin{enumerate}[label=(\alph*)]
\item $\mu=\nu$ and $M_{k}$ is a scalar matrix;
    then, for any $\ell>0$, $M_k^\ell$ is
    \begin{align*}
        M_{k}^\ell&= \begin{pmatrix}  \mu^\ell & 0\\ 0 & \mu^\ell\end{pmatrix}.
    \end{align*}
    Now, for any $\ell>0$, the quotient of the first column, i.e.\ the $k\ell$-th convergent of $\mathcal{P}$ is $\infty$, while the quotient of the second column, i.e.\ the preceding convergent, is~$0$, which contradicts the assumption of convergence.
    \item  $\mu=\nu$ and  $M_{k}$ is conjugate to $\mu\cdot \begin{pmatrix}  1 & 1\\ 0 & 1\end{pmatrix}$.
    Let $\mathcal{A}=\begin{pmatrix} x&y\\z&w\end{pmatrix}$ be a matrix such that 
    $$M_{k}=\mathcal{A}^{-1}\mu\cdot \begin{pmatrix}  1 & 1\\ 0 & 1\end{pmatrix}\mathcal{A};$$
    then, for any $\ell>0$,
    \begin{align*}
        M_{k}^\ell&=\mu^\ell \mathcal{A}^{-1}\cdot \begin{pmatrix}  1 & \ell\\ 0 & 1\end{pmatrix}\mathcal{A}\\
        &= \frac {\mu^\ell} {\det(\mathcal{A})} \begin{pmatrix} \det(\mathcal{A})+ zw\ell & w^2 \ell\\ -z^2\ell & \det(\mathcal{A})- zw\ell\end{pmatrix}.
    \end{align*}
    Now, for $\ell=p^h$, the quotient of the first column, i.e.\ the $kp^h$-th convergent of $\mathcal{P}$, tends to $\infty$, while the quotient of the second column, i.e.\ the preceding convergent, tends to $0$, which contradicts the assumption of convergence.
    
   \item \label{itm:convergence}if $\mu\not=\nu$, then $M_{k}$ is diagonalizable. We first prove that $B_k\not =0$; indeed, if this is not the case, then  we would have $$M_k=\begin{pmatrix} \mu & \delta \\ 0 &\nu\end{pmatrix},$$ so that, for $\ell>0$,
    \begin{equation}
        \label{eqn:eigenvalues}
        M_k^\ell =\begin{pmatrix} \mu^\ell & \delta\sum_{i=0}^{\ell-1}\mu^i\nu^{\ell-1-i} \\ 0 &\nu^\ell\end{pmatrix}.
    \end{equation} Therefore the $k\ell$-th convergent is $\infty$, while the $(k\ell-1)$-th convergent has a $p$-adic valuation $\geq v_p(\delta)$ (recall that $|\mu|_p=|\nu|_p=1$), contradicting the assumption of convergence. \\Therefore the second component of each eigenvector is not zero, so that we can write two linearly independent eigenvectors as 
   \begin{equation}\label{eq:eigenvectors} \begin{pmatrix} \beta \\ 1\end{pmatrix},\quad\quad \begin{pmatrix} \beta^* \\ 1\end{pmatrix}\end{equation}
   with $\beta, \beta^* \in \mathbb{Q}_p$.\\
    Then,
    \begin{equation}\label{eq:Mkell} \begin{split}
        M_{k}^\ell& =\begin{pmatrix} \beta & \beta^*\\ 1 & 1 \end{pmatrix} \begin{pmatrix} \mu^\ell & 0\\ 0 & \nu^\ell\end{pmatrix} \begin{pmatrix} \beta & \beta^*\\ 1 & 1 \end{pmatrix}^{-1}\\
        &=\frac 1 {\beta-\beta^*}\begin{pmatrix}\beta \mu^\ell-\beta^*\nu^\ell & -\beta\beta^*(\mu^\ell-\nu^\ell)\\
        \mu^\ell-\nu^\ell & \nu^\ell\beta-\mu^\ell\beta^*
        \end{pmatrix}.
    \end{split}
    \end{equation}
The ratio of the first column is
$\frac{\beta(\mu\nu^{-1})^\ell-\beta^*}{(\mu\nu^{-1})^\ell-1}$. By assumption this ratio $p$-adically converges for $\ell\to \infty$. Then, since $\beta\not=\beta^*$ also $(\mu\nu^{-1})^\ell$ $p$-adically converges, but this implies $|\mu\nu^{-1}|_p<1$, a contradiction.
\end{enumerate}
We just proved that if $\mathcal{P}$ is convergent then $|A_k+B_{k-1}|_p>1$. \\

   Now we prove that, if $\mathcal{P}$ is convergent, then $ii$) holds. Since every $\mathcal{P}_j$ is convergent, it suffices to prove that, if $B_k=0$, then $|B_{k-1}|_p<1$.  If this is not the case, then, by point $i$) and the above remarks on the eigenvalues of $M_k$ (which are in this case $A_k$ and $B_{k-1}$),  we would have 
$|B_{k-1}|_p> 1$ and $|A_{k}|_p< 1 $, so that, for every $\ell>0$,
$$\frac{A_{\ell k}}{B_{\ell k}}=\frac{A_{ k}^\ell}{0}=\infty,$$
and, by~\eqref{eqn:eigenvalues},
$$\frac{A_{\ell k-1}}{B_{\ell k-1}}=\frac{A_{k-1} (A_{ k}^\ell - B_{k-1}^\ell)}{B_{k-1}^\ell(A_k-B_{k-1})}\buildrel p \over \longrightarrow - \frac{A_{k-1}}{A_k-B_{k-1}} \hbox{ for } \ell\to \infty; $$
therefore $\left|\frac{A_{\ell k-1}}{B_{\ell k-1}}\right |_p\leq \left |\frac{A_{k-1}}{B_{k-1}}\right|_p<\infty$,  contradicting the convergence of $\mathcal{P}$.

Conversely, assume that 
conditions $i)$ and $ii)$ hold. Since the matrices $T_j$  have the same characteristic polynomials, they have the same eigenvalues, say $\mu$ and $\nu$, satisfying $|\mu + \nu|_p>1$  and $\mu\nu=(-1)^k$. This implies that $\mu \neq \nu$ and 
$|\mu|_p>1$, $|\nu|_p<1$. For every $j$, let $\beta_j\in\PP^1(\overline{\QQ}_p)$ be the fixed point of $T_j$ associated to the dominant eigenvalue $\mu$. Then  it is easy to see that $M_{j-1}T_j=T_1M_{j-1}$ (defining $M_0$ to be the identity matrix), so that 
$\overline M_{j-1}\beta_j=\beta_1$. Now let $T$ be any of the matrices $T_j$, and $\beta=\beta_j$. Let $\beta^*$ be the other fixed point.
We claim that the ratios of the entries of the first column of $T^\ell$ converge $p$-adically to $\beta$, that is $\overline{T}^\ell(\infty)=\beta$. We consider  two cases separately.
\begin{enumerate} [label=(\alph*)]
\item[a)] Neither of the two eigenspaces of $T$ is the line  $y=0$. Therefore we can choose two linearly independent eigenvectors  as in \eqref{eq:eigenvectors}. As in \eqref{eq:Mkell} we have the equality
\begin{align*}
        T^\ell& =\frac 1 {\beta-\beta^*}\begin{pmatrix}\beta \mu^\ell-\beta^*\nu^\ell & -\beta\beta^*(\mu^\ell-\nu^\ell)\\
        \mu^\ell-\nu^\ell & \nu^\ell\beta-\mu^\ell\beta^* 
        \end{pmatrix}
    \end{align*}
    for each $\ell>0$. In this case, recalling that $\nu^\ell\buildrel p\over\rightarrow 0$, the ratio of the entries of the first column converges to $\beta$, i.e.\ 
    $$\frac{\beta \mu^\ell-\beta^*\nu^\ell}{\mu^\ell-\nu^\ell}\buildrel p\over\rightarrow \beta.$$
  \item[b)] One of the two eigenspaces is the line $y=0$. By the hypothesis $ii)$, we must have 
  \begin{equation*} T= \begin{pmatrix}\mu & *  \\0
         & \nu
        \end{pmatrix}.
        \end{equation*}
        Therefore   $\overline{T}^\ell(\infty)=\infty=\beta$.
    \end{enumerate}

    For any $j \in \{1,\dots,k\}$ and $\ell\in\NN$, $M_{\ell k+j-1} =M_{j-1}T_j^\ell$; therefore

    \begin{equation*}
    \overline{M}_{\ell k+j-1}(\infty)=\overline{M}_{j-1}\overline{T}_j^\ell(\infty)\buildrel p \over \rightarrow \overline M_{j-1}(\beta_j)=\beta_1.\qedhere
    \end{equation*}
      \end{proof}

\begin{remark}\label{rem.condizione.ii}
    Condition (\textit{ii}) in Theorem~\ref{criterio.di.convergenza.padica} is necessary to obtain convergence of the continued fraction. In the following example, convergents tend to different limits along different arithmetic progressions.
    
Consider the $\PPCF$ $[\overline{1,-\frac 1 p,p}]$; in this case $k=3$ and
\begin{align*}T_1=M\left(1,-\frac 1 p,p\right)&=\begin{pmatrix}p & 1-\frac{1}{p}  \\0
         & -\frac{1}{p}
        \end{pmatrix},\\T_2= M\left(-\frac 1 p,p,1\right)&=\begin{pmatrix}-\frac{1}{p} & 0  \\p+1
         & p
        \end{pmatrix}, \\ T_3= M\left(p,1,-\frac 1 p\right)&=\begin{pmatrix}p-1-\frac{1}{p} & p+1  \\1-\frac{1}{p}
         & 1
        \end{pmatrix}.
\end{align*}
 Then $|A_k+B_{k-1}|_p=|\mathrm{Tr}(T_1)|_p=|p-\frac{1}{p}|_p=p>1$, so that (\textit{i}) is satisfied, but $T_1$ does not satisfy condition (\textit{ii}).

It is easy to check by induction that
\begin{align}
    A_{3n}&=p^n, & A_{3n+1}&=\frac{p+1}{p^2+1}p^{n+1}-\frac{p-1}{p^2+1}\left(\frac{-1}{p}\right)^{n}, & A_{3n+2}&=\frac{p-1}{p^2+1}\left(p^{n+1}-\left(\frac{-1}{p}\right)^{n+1}\right)\\
    B_{3n}&=0, & B_{3n+1}&=\left(\frac{-1}{p}\right)^n, & B_{3n+2}&=\left(\frac{-1}{p}\right)^{n+1}
\end{align}
and therefore the sequence $[a_1,\dotsc,a_{m}]$ does not have a limit; in fact the quantity $[a_1,\dotsc,a_{m}]$ is $\infty$ for $m$ multiple of 3 and 
tends to $(1-p)/(p^2+1)$ for $m$ not multiple of 3. 
\end{remark}

\begin{remark}\label{rmk:convnotpreserved}
    The isomorphism $\sigma_k$ defined in Proposition~\ref{prop:scambio} does not preserve convergence. Considering again the example from Remark~\ref{rem.condizione.ii}, the $\PPCF$ $[\overline{1,-\frac 1 p,p}]$ corresponds to a non-converging point in  $V(F)_{0,3}$ for $F=(p^2+1)x+(p-1)$, while $\sigma_3(1,-\frac{1}{p},p)= (p,-\frac{1}{p},1)\in V(G)_{0,3}^{\mathrm{con}}$ for $G=(p-1)x^2 - (p^2+1)x$, as can be easily seen applying Theorem~\ref{criterio.di.convergenza.padica}. The value of the limit can be shown to be zero.
\end{remark}

\section{Families of PCF of type \texorpdfstring{$(N, k)$}{(N,k)} when \texorpdfstring{$N + k \leq  3$}{N+k<=3}}
\label{sec:familiesPCF}

 Let $[A:B:C]$ be a point in $\mathbb{P}^2(\QQ)$ and consider the (non-trivial) quadratic polynomial equation $F(x)=0$, where
$$F(x)=Ax^2+Bx+C\in  \QQ[x].$$
Let $\mathcal{B}=\{\beta, \beta^*\}\subseteq \mathbb{P}^1(\overline{\mathbb{Q}})$ be the multi-set of roots of  $F(x)$. If $A\not=0$, then $\beta,\beta^\ast\in\overline{\QQ}$, and
\[\beta+\beta^\ast=-\frac B A,\quad\quad \beta\beta^*= \frac C A.\]\\
 In this section, we consider the same types as in~\cite[§5-8]{BrockElkies2021}, i.e.\ the types such that $V(F)_{N,k}$ is either zero or one dimensional, and study them in the $p$-adic setting. In particular, we will characterize the cases in which $V(F)^{\mathrm{con}}_{N,k}$ is non-empty and  $\beta, \beta^*$ are rationals or square roots of rationals.
 \begin{remark} Let $\Delta=B^2-4AC$ be the discriminant of $F$. Notice that if  $\Delta\not\in \QQ_p^2$ then $ V(F)^{\mathrm{con}}_{N,k}=\varnothing$ for every $N,k$. Indeed the limit of a PCF having partial quotients in $\calo$ must belong to $\PP^1(\QQ_p)$. 
 \end{remark}
 
\subsection{Type \texorpdfstring{$(0,1)$}{(0,1)}}
\label{sect:01}
The following result follows immediately from Table~\ref{tab:polynomials}, equation~\eqref{eq:VF}, and  Theorem~\ref{criterio.di.convergenza.padica}.
\begin{proposition}\ 
\label{prop:PCF01}
\begin{enumerate}[label=(\alph*)]
    \item \label{itm:V_01} $V(F)_{0,1}(\calo)$ is non-empty if and only if $A=-C\neq 0$ and $B/A \in \calo$. If this is the case, then $V(F)_{0,1}(\calo)=\{-B/A\}$.
    \item\label{itm:Vconv_01} $V(F)^{\mathrm{con}}_{0,1}$ is non-empty if and only if  $V(F)_{0,1}(\calo)$ is non-empty and $ |B|_p>|A|_p$. If this is the case, then $V(F)_{0,1}(\calo)= V(F)^{\mathrm{con}}_{0,1}$ and $[\overline{-B/A}]$ converges to the unique root of $F$ with $p$-adic absolute value $>1$.
\end{enumerate}
    
\end{proposition}

The following proposition characterizes the convergent locus of a polynomial which  is reducible over $\QQ$.

 \begin{proposition}
\label{prop:rational01}
Suppose that $F$ is reducible over $\QQ$. Then, $V(F)^{\mathrm{con}}_{0,1}$ is non-empty if and only if the multi-set of the roots of $F$ is $\mathcal{B}=\{\pm 1/{p^\alpha}, \mp p^\alpha\}$ for some $\alpha>0$ and $V(F)^{\mathrm{con}}_{0,1}=\{(\pm(1-p^{2\alpha})/p^\alpha)\}$.
\end{proposition} 

\begin{proof}
We already noticed that $V(F)_{0,1}(\calo)$ is non-empty if and only if  $A=-C\not=0$ and $B/A\in\calo$.
A quadratic monic polynomial with coefficients in $\calo$ and constant term equal to $-1$ is reducible over $\QQ$ if and only if its roots are units in $\calo$ with product $-1$ (because $\calo$ is integrally closed); therefore the set of roots of $F$ is $$\{\pm 1/{p^\alpha}, \mp p^\alpha\}$$ for some $\alpha\in\NN$.
The sum \[-\frac{B}{A}= \pm\frac{1}{p^\alpha} \mp p^\alpha=\pm \frac{1-p^{2\alpha}}{p^\alpha}\]
has a $p$-adic absolute value greater than 1 if and only if $\alpha>0$.
\end{proof}

\begin{remark}
\label{rem:bedocchi01}
     We observe that no irrational square root can be the limit of a PCF of type~$(0,1)$. This is consistent with Bedocchi's result on Browkin continued fractions, stating that the Browkin continued fraction expansion of an irrational square root of an integer, if periodic, must have preperiod of length $2$ or $3$~\cite[§3]{Bedocchi1988}.
\end{remark} 
 
 Finally, we observe that, by part~\textit{\ref{itm:V_01}} of Proposition~\ref{prop:PCF01}, $V(F)_{0,1}(\calo)$ is empty  for every polynomial $F(x)=Ax^2+C$, except for the reducible case $A=-C$.

\subsection{Type \texorpdfstring{$(1,1)$}{(1,1)}} 
\label{sect:11}
The convergence of a PCF of type $(1,1)$ depends only on its purely periodic part and can be therefore studied in the light of type $(0,1)$.

Using Table~\ref{tab:polynomials} and \eqref{eq:VF}, the equations of $V(F)_{1,1}$ are
\begin{equation}\label{eq:1-1}
\begin{cases}  
Aa_1 - 2Ab_1 - B &=0, \\
 Ab_1^2 + Bb_1 + A + C &=0.
\end{cases}  
\end{equation}

\begin{proposition}\
\label{prop:PCF11}
\begin{enumerate}[label=(\alph*)]
    \item \label{itm:VF11}  $V(F)_{1,1}(\calo)$ is not empty if and only if $A \neq 0$, $B/A, C/A \in \calo$ and $B^2-4A(A+C)$ is a square in $\QQ$. If this is the case, then $(B^2-4A(A+C))/A^2$ is a square in $\calo$ and $ V(F)_{1,1}(\calo)=\{(b_1,2b_1+B/A), (-B/A-b_1,-2b_1-B/A)\}$ where $b_1$ is a root of $F(x)+A$, and $\beta\neq\beta^*$.
    \item \label{itm:VFcon11} $V(F)_{1,1}^{\mathrm{con}}$ is not empty if and only if $V(F)_{1,1}(\calo)$ is not empty and  $|B^2-4A(A+C)|_p>|A^2|_p$.  If this is the case, then $V(F)_{1,1}^{\mathrm{con}}=V(F)_{1,1}(\calo) $ and one of the two associated continued fractions converges to $\beta$, while the other converges to $\beta^*$.
\end{enumerate}

\end{proposition}
\begin{proof}
    Part~\textit{\ref{itm:VF11}} is easily deduced by studying the solutions of~\eqref{eq:1-1}. In particular, if  $V(F)_{1,1}(\calo)$ is not empty, then $A\neq 0$ (otherwise also $B=C=0$, a contradiction) and $\beta\neq \beta^*$ since, otherwise, $\Delta=0$ and $B^2-4A(A+C)$ being a square would imply that $-4A^2$ too is a square in $\QQ$.
    
    We prove now~\textit{\ref{itm:VFcon11}}. Let $(b_1,a_1)$ be a point in $V(F)_{1,1}(\calo)$ and suppose that $|B^2-4A(A+C)|_p>|A^2|_p$. Since $\left(B^2-4A(A+C)\right)/A^2=a_1^2$, we have $|a_1|_p>1$, which in turn implies that $[\overline{a_1}]$ converges by part~\textit{\ref{itm:Vconv_01}} of Proposition~\ref{prop:PCF01}. Hence $(b_1,a_1)\in V_{1,1}^{\mathrm{con}}(F)$.
    Conversely, suppose that $(b_1,a_1) \in V_{1,1}^{\mathrm{con}}(F)$. Then $b_1=\frac{Aa_1-B}{2A}$ and both $b_1$ and $b_1-a_1$ are roots of $F(x)+A$, whose discriminant is $B^2-4A(A+C)=(a_1A)^2$. 
    Finally, the convergence of the $\PPCF$ $[\overline{a_1}]$ implies $|B^2-4A(A+C)|_p>|A^2|_p$. In this case, $V_{1,1}^{\mathrm{con}}(F)$ also contains the point  $(-B/A-b_1, -a_1)$. To prove this, observe that the purely periodic continued fraction $[\overline{a_1}]$ converges to some $\alpha$ which is---as we saw for type $(0,1)$---the $p$-adically dominant root of the polynomial $x^2-a_1x-1$. It follows that  $-\alpha$ is the limit of the continued fraction $[\overline{-a_1}]$. Therefore, if $\beta$ is the limit of the continued fraction $[b_1,\overline{a_1}]$, then $[-B/A-b_1,\overline{-a_1}]$ converges to $-B/A-b_1-\frac 1 \alpha=-B/A-\beta=\beta^*$, so that $(-B/A-b_1,-a_1) \in V_{1,1}^{\mathrm{con}}(F)$.
\end{proof}

When $A\not=0, B=0$, the above proposition provides information about the square roots of rationals which can be expressed as $p$-adically convergent continued fractions of type~$(1,1)$. 
\begin{corollary}
\label{cor:rational11}
Suppose $A\not=0, B=0$; then, $V(F)^{\mathrm{con}}_{1,1}$ is non-empty if and only if there exists $b_1 \in \calo$ with $|b_1|_p>1$ such that $-C/A=b_1^2+1$.
In this case $V(F)^{\mathrm{con}}_{1,1}=\{\pm(b_1,2b_1)\}$.

\end{corollary}
This point is the $p$-adic analogue of the classical expansion $\sqrt{n^2+1}=[n,\overline{2n}]$ with $n\in\ZZ_{>0}$.


\subsection{Type \texorpdfstring{$(2,1)$}{(2,1)}}
\label{sect:21}
The equations of $V(F)_{2,1}$ are

\begin{equation}
\label{eq:riscrittura2}
\begin{cases} A(-2a_1 b_1 b_2 +2 b_1b_2^2 -a_1-2 b_1+2 b_2)-B(b_2 a_1 -b_2^2+1) &=0,\\ 
A(a_1 b_1^2 b_2 - b_1^2 b_2^2 +a_1 b_1 + b_1^2 - 2 b_2 b_1 -1) - C(b_2 a_1 -b_2^2+1)&=0,\\
B(a_1 b_1^2 b_2 - b_1^2 b_2^2 +a_1 b_1 + b_1^2 - 2 b_2 b_1 -1)- C(-2a_1 b_1 b_2 +2 b_1b_2^2 -a_1-2 b_1+2 b_2)&=0.
\end{cases} 
\end{equation}

If $A\neq 0$ the third equation is redundant, but it is necessary if $A=0$.
By eliminating $a_1$ we see that the third equation may be replaced by

\begin{equation}
\label{eq:riscrittura3}
A(b_1^2b_2^2+ 2b_1b_2 + b_1^2+1) + B(b_1b_2^2+b_1+b_2) +  C(b_2^2 + 1) =0.
\end{equation}

\begin{remark}
    If $A=B=0$ the third equation implies that $b_2^2+1=0$ which has no solutions in $\calo$; in this case $V(F)_{2,1}(\calo)$ is empty.
\end{remark}

\begin{proposition}
\label{prop:21a}
 Suppose that $A=0$ and $B,C$ are coprime integers with $B>0$. Then $V(F)_{2,1}(\calo)$ is non-empty if and only if there exist $\epsilon_1\in \{1,-1\}$, an integer $k$, and $\alpha,\beta \in \NN$, such that
    \[B= p^\beta(p^{2\alpha}+1)\qquad \text{and} \qquad
C=k(p^{2\alpha}+1)-\epsilon_1 p^{\alpha+\beta}.\ \]
If this is the case and $\alpha>0$, then $\alpha,\beta,k,\epsilon_1$ are uniquely determined, $V(F)_{2,1}(\calo)=V(F)_{2,1}^{\mathrm{con}}$ and it consists of two distinct points $(b_1,b_2,a_1)$ given by
\begin{align*}
b_1&=- p^{-\beta}k & b_1&=- p^{-\beta}k\\
b_2&=\epsilon_1 p^{ \alpha} & b_2&=\epsilon_1 p^{ -\alpha}\\
a_1&=\epsilon_1 (p^\alpha-p^{-\alpha}) & a_1&=-\epsilon_1 (p^\alpha-p^{-\alpha});
\end{align*}
the two associated continued fractions converge to the point at infinity and to $-C/B$ respectively.

If $\alpha =0$ then $V(F)_{2,1}(\calo)$ consists of the two points
\begin{align*}
b_1&=-\frac{C+p^\beta}{2 p^\beta} & b_1&=-\frac{C-p^\beta}{2 p^\beta}\\
b_2&=1 & b_2&=-1\\
a_1&=0 & a_1&=0
\end{align*}
and $V(F)_{2,1}^{\mathrm{con}}=\varnothing$.

\end{proposition}
\begin{proof}
The first equation gives
\[
a_1=\frac{b_2^2-1}{b_2}=b_2-\frac{1}{b_2}.
\]
This implies that $b_2$ is invertible in $\calo$ and therefore $b_2=\epsilon_1 p^{\epsilon_2 \alpha}$ for some $\epsilon_1,\epsilon_2\in\{1,-1\}$ and $\alpha\in\NN$, so that
\[
a_1=\epsilon_1 (p^{\epsilon_2\alpha}-p^{-\epsilon_2\alpha})=\epsilon_1\epsilon_2 (p^{\alpha}-p^{-\alpha}).
\]
The third equation now gives
\[
b_1=-\frac{C(b_2^2+1)+Bb_2}{B(b_2^2+1)}=-\frac{C(p^{2\alpha}+1)+B\epsilon_1 p^\alpha}{B(p^{2\alpha}+1)},
\]
where the second equality is obtained multiplying numerator and denominator by $p^{2\alpha}$ if $\epsilon_2=-1$. Then $p^{2\alpha}+1$ divides $B$, so write $B=r(p^{2\alpha}+1)$ and we get
\[
b_1=-\frac{C+\epsilon_1 r p^\alpha}{r(p^{2\alpha}+1)};
\]
this shows that $r$ is not divisible by any prime different from $p$, because $B$ and $C$ are coprime, and that $p^{2\alpha}+1$ divides $C+\epsilon_1 r p^\alpha$. Write $r=p^\beta$ with $\beta\in\NN$, and $C+\epsilon_1 r p^\alpha=k(p^{2\alpha}+1)$ for some $k\in\ZZ$.
Therefore we get the family of solutions
\begin{align*}
B&=p^\beta(p^{2\alpha}+1)\\
C&=k(p^{2\alpha}+1)-\epsilon_1 p^{\alpha+\beta}\\
a_1&=\epsilon_1\epsilon_2 (p^{\alpha}-p^{-\alpha})\\
 b_1&=- p^{-\beta}k\\
b_2&=\epsilon_1 p^{\epsilon_2 \alpha}
\end{align*}
where $\alpha,\beta\in\NN,k\in\ZZ, \epsilon_1,\epsilon_2 \in\{\pm 1\}$.

By Proposition~\ref{prop:PCF01} the associated continued fraction converges if and only if $|a_1|_p>1$, which happens if and only if $\alpha>0$; the value of the limit follows from Proposition~\ref{prop:PCF01} as well.
\end{proof}

The next proposition deals with the case $A\not=0$. 

\begin{proposition}
\label{prop:21b}
Suppose that $A \neq 0$. Then,
\begin{enumerate}[label=(\alph*)]
    \item $V(F)_{2,1}(\calo)$ consists of finitely many points. In particular, it is empty if $\Delta\leq 0$ or $\Delta=4A^2$.
    \item $V(F)_{2,1}^{\mathrm{con}}$ consists of the points $(b_1,b_2,a_1)\in V(F)_{2,1}(\calo)$ such that \[|2Ab_1b_2^2+Bb_2^2-2Ab_1+2Ab_2-B|_p>|2Ab_1b_2+Bb_2+A|_p.\]
\end{enumerate}

\end{proposition}
\begin{proof}
Assume that $(b_1,b_2,a_1)\in V(F)_{2,1}(\calo)$.
First, notice that, from the first equation of \eqref{eq:riscrittura2}, we get that 
\begin{equation} \label{eq:a1} a_1=\frac{2Ab_1b_2^2+Bb_2^2-2Ab_1+2Ab_2-B}{2Ab_1b_2+Bb_2+A}. \end{equation}
Let us prove that the denominator is always nonzero. Indeed, assume by contradiction that $2Ab_1b_2+b_2B+A=0$; then, if we subtract this denominator from \eqref{eq:riscrittura3}, we get 
\begin{align*}
0& =A(b_1^2b_2^2+ 2b_1b_2 + b_1^2+1) + B(b_1b_2^2+b_1+b_2) +  C(b_2^2 + 1) - (2Ab_1b_2+b_2B+A )= \\
& = (b_2^2+1)(Ab_1^2+Bb_1+C).
\end{align*}
Notice that, since $b_2 \in \calo$, we have that $b_2^2+1 \neq 0$; this implies that $F(b_1)=Ab_1^2+Bb_1+C=0$.
But now if we look at the expression of $\mathrm{Quad}(\mathcal{P})$ in Table \ref{tab:polynomials} and we evaluate it in $b_1$ we get the value $-1$; since the polynomial $F(x)=Ax^2+Bx+C$ is a multiple of $\mathrm{Quad}(\mathcal{P})$, we have that $F(b_1)=0$ if and only if $A=B=C=0$, which is not the case.\\
Let $C_1$ be the projective plane curve defined by \eqref{eq:riscrittura3}.
A straightforward calculation shows that for  $\Delta(\Delta-4A^2)\neq 0$ the map $b_2:C_1\to \mathbb{P}^1$ ramifies in exactly  four points, so that by the Riemann-Hurwitz formula $C_1$ has genus 1. This implies in particular that it has finitely many points over $\calo$.

Furthermore we notice that the discriminant of \eqref{eq:riscrittura3} with respect to the variable $b_1$ is the polynomial $\Delta(b_2^2+1)^2-4A^2$.
If $\Delta\leq 0$, then this discriminant is negative, and therefore equation \eqref{eq:riscrittura3} doesn't have real solutions.

If $\Delta= 4A^2$, then the discriminant of  \eqref{eq:riscrittura3} is equal to $4 A^2 b_2^2(b_2^2+2)$. In order for equation \eqref{eq:riscrittura3} to have rational solutions this discriminant must be the square of a rational number, so that, after dividing by $4 A^2 b_2^2$ we see that $b_2^2+2=v^2$ for some $v\in\QQ$. But then $v$ must be in $\calo$ and, after clearing denominators, we obtain $n^2+2 p^{2\alpha}=m^2$ for $n,m\in\ZZ$, $\alpha\in\NN$. Since we are assuming $p$ odd, this equation doesn't have solutions modulo 4.
This concludes the proof of part $(a)$.\\
Concerning  $(b)$, we observe that  the convergence of a PCF of type $(2,1)$ depends only on its purely periodic part and can be therefore studied in the light of type $(0,1)$. Then a point $(b_1,b_2,a_1)\in V(F)_{2,1}(\calo)$ belongs to $V(F)_{2,1}^{\rm{con}}$ if and only if 
 $|a_1|_p>1$, by Theorem~\ref{criterio.di.convergenza.padica}.
 Then the assertion follows from the expression of $a_1$ in formula \eqref{eq:a1}.  
\end{proof}

\subsection{Type \texorpdfstring{$(0,2)$}{(0,2)}} \label{sect:02}
The equations of $V(F)_{0,2}$ are
\begin{equation*}\begin{cases} -Aa_1 a_2-Ba_2 &=0,\\ -A a_1 - C a_2&=0,\\ C a_1 a_2 -B a_1 &=0.
\end{cases}   \end{equation*}

\begin{proposition}\ 
\label{prop:PCF02}
\begin{enumerate}[label=(\alph*)]
\item  All the $\calo$-integer points of $V(F)_{0,2}$ are given by:
\label{itm:VF02} \begin{equation*}
V(F)_{0,2}(\calo)=
\begin{cases}
\{(a_1,0)\mid a_1\in\calo \} & \text{if }A=B=0,\\ 
\{(0,a_2)\mid a_2\in\calo \} & \text{if }A\neq 0, B=C=0,\\
 \{(0,0), (- \frac B A, \frac B C)\}&\text{if }ABC\neq 0\text{ and }\frac{B}{A},\frac{B}{C}\in\calo,\\
\{ (0,0)\} & \text{otherwise.}
\end{cases}
\end{equation*}
\item \label{itm:VF02con} $V(F)^{\mathrm{con}}_{0,2}$ is non-empty if and only if $ABC\neq 0$, $\frac{B}{A},\frac{B}{C}\in\calo$ and $|B|^2_p > |AC|_p$. In this case $V(F)^{\mathrm{con}}_{0,2}=\left \{(- \frac B A, \frac B C) \right \}$; in particular, when $A=-C$, the convergent continued fraction has in fact type $(0,1)$. 
\end{enumerate}
\end{proposition}
\begin{proof}
\ref{itm:VF02} follows from direct computations, while~\ref{itm:VF02con} follows from Theorem~\ref{criterio.di.convergenza.padica}.
\end{proof}
\begin{remark}
\label{rem:bedocchi02}
     We observe that no irrational square root can be the limit of a PCF of type~$(0,2)$. As for the type $(0,1)$ (Remark~\ref{rem:bedocchi01}), this is consistent with Bedocchi's result.
\end{remark}

By Theorem~\ref{criterio.di.convergenza.padica}, the purely periodic continued fraction $[\overline{a_1a_2}]$ converges if and only if $|a_1a_2|_p > 1$. Remarkably, this implies that $a_1a_2(a_1a_2+4)$ is a square by Hensel's Lemma. 

\par We notice that $(a_1,a_2)\in V(F)_{0,2} $ if and only if $(a_2,a_1) \in V(G)_{0,2}$ where $G(x)=Cx^2-Bx+A$.

When $F$ is reducible over $\QQ$, we have the following.

\begin{proposition}
\label{prop:02-Reducibility}
{Suppose that $F$ is reducible over $\mathbb{Q}$ and $V(F)^{\mathrm{con}}_{0,2}$ is non-empty.} Then, $V(F)^{\mathrm{con}}_{0,2}=\{(a_1,a_2)\}$ where the product $a_1a_2$ satisfies
\begin{equation*}
    a_1a_2=\pm \frac{(p^k-1)^2}{4^\epsilon p^{k}},
\end{equation*}
with $k \in \mathbb{Z}_{>0}$ and $\epsilon \in \{0,1\}$.
\end{proposition} 
\begin{proof}
Since we are assuming $V(F)^{\mathrm{con}}_{0,2}$ non-empty, we have that $ABC\neq 0$ and $\frac{B}{A},\frac{B}{C}\in\calo$. 
$F$ splits over $\mathbb{Q}$ if and only if
its discriminant is a square. Recalling that in this case we have $a_1=-\frac{B}{A}$ and $\frac{a_1}{a_2}=-\frac{C}{A}$, the discriminant of $F$ becomes
\begin{equation*}
    \Delta=A^2\left (a_1^2+4\frac{a_1}{a_2} \right );
\end{equation*}
by multiplying by $a_2^2$ we have that $\Delta$ is a square if and only if there exist $r,s \in \ZZ$ coprime with $p$ and $k,\ell >0$ with $a_1a_2=s/p^k$ such that
\begin{equation}
    \label{eqn:disc11}
    \frac{s}{p^{k}}\left(\frac{s}{p^{k}}+4 \right)=\frac{r^2}{p^{2\ell}}.
\end{equation}
Now, it is easy to see that $k$ and $\ell$ must be equal, so that~\eqref{eqn:disc11} becomes
\begin{equation*}
    s(s+4p^{k})=r^2.
\end{equation*}
Since $s$ is coprime with $p$, the only possible common divisors between the two factors above are, up to a sign, $2$ and $4$; we consider the three cases separately.
\begin{itemize}
    \item[(i)] If $s$ and $s+4p^k$ are coprime, then $s$ is odd and we can write
\begin{align*}
    s=t^2 &&\text{and} && s + 4 p^{k}=u^2
\end{align*}
for some coprime odd integers $t$ and $u$ not divided by $p$.
These two equations yield
\begin{align*}
    4p^{k}=(u+t)(u-t),
\end{align*}
and the only possibility, up to the signs of $u$ and $t$, is $u+t=2p^{k}$ and $u-t=2$. Therefore $s=(p^k-1)^2$.
\item[(ii)] If the greatest common divisor between $s$ and $s+4p^k$ is $2$, then $s=2s'$ for some odd $s'$ and we can write
\begin{align*}
    s'=t^2 &&\text{and} && s' + 2 p^{k}=u^2
\end{align*}
for some coprime odd integers $t$ and $u$ not divided by $p$. However, this leads to a contradiction since $(u+t)(u-t)=2p^k$ implies that only one between $(u+t)$ and $(u-t)$ is even, which is clearly false.
\item[(iii)] If the greatest common divisor between $s$ and $s+4p^k$ is $4$, then $s=4s'$ for some odd $s'$, then we can write
\begin{align*}
    s'=t^2 &&\text{and} && s' + p^{k}=u^2
\end{align*}
for some coprime odd integers $t$ and $u$ not divided by $p$. These two equations yield
\begin{align*}
    p^{k}=(u+t)(u-t),
\end{align*}
and the only possibility, up to the signs of $u$ and $t$, is $u+t=p^{k}$ and $u-t=1$. Therefore $s=((p^k-1)/2)^2$.\qedhere
\end{itemize}
\end{proof}


\subsection{Type \texorpdfstring{$(1,2)$}{(1,2)}}
\label{sect:12}
\label{subsec:12}
The equations of $V(F)_{1,2}$ are
\begin{equation}\label{eq:1-2}\begin{cases} -Ba_1 +A( a_1a_2 - 2a_1b_1) &=0,\\ 
A(-a_1a_2b_1 + a_1b_1^2 - a_2)-Ca_1 &=0,\\
B(-a_1a_2b_1 + a_1b_1^2 -a_2) +C (-a_1 a_2 + 2a_1b_1)&=0.
\end{cases}   \end{equation}

We notice that  equations \eqref{eq:1-2} imply the relation
\begin{equation} \label{eq:1-2bis}
A^2a_2(a_1a_2+4)-(B^2-4AC)a_1=0,
\end{equation}
involving only $a_1,a_2$.

If $A\neq 0$ and $-B/A \in \calo$, then the map 
\begin{equation}
    \label{eqn:inv}
    (b_1,a_1,a_2)\mapsto (-b_1-B/A,-a_1,-a_2)
\end{equation}
is an involution of $V(F)_{1,2}$.

To ease the notation, we will denote by $R_1$ the line $a_1=a_2=0$, which is always contained in $V(F)_{1,2}$.

We study $V(F)_{1,2}(\calo)$ by considering different cases.
\begin{itemize}
    \item \textbf{If $\bm{A=0}$.} We have
    \begin{align*}
        V(F)_{1,2}(\calo)=\begin{cases}
            \{(b_1,0,a_2)\ |\ b_1,a_2\in \calo\} &\qquad \text{if $B=0$},\\
            R_1(\calo)  &\qquad \text{if $B\neq 0$.}
        \end{cases}
    \end{align*}
    In both cases, according to the results in Section \ref{sect:02}, none of the continued fractions of the form \([\overline{0,a_2}]\) $p$-adically converge, so $V(F)_{1,2}^{\mathrm{con}}=\varnothing$.
    \item\textbf{If $\bm{A\neq 0}$ and $\bm{B^2-4AC=0}$.}  By \eqref{eq:1-2bis}, either $a_2=0$ or $a_1a_2+4=0$. In the first case, by \eqref{eq:1-2}, we have that $a_1(B+2Ab_1)=0$, giving the two sets 
$R_1(\calo)$ and $R_2(\calo)$, where $R_2$ is the line of equation $b_1+B/2A=a_2=0$. We note that $R_2(\calo)\neq \varnothing$ only if $-B/2A \in \calo$.

In the second case, we have that $a_1 =-\frac{4}{a_2}$ and, by \eqref{eq:1-2} we have that $b_1=\frac{a_2-B/A}{2}$, hence we have the set
$\{ (\frac{a_2-B/A}{2}, \frac{-4}{a_2} , a_2)\ |\ a_2 \in \QQ \} \cap\calo^3$.
Hence we get
\begin{align*}
  V(F)_{1,2}(\calo)& =R_1(\calo) \cup R_2(\calo) \cup \left( \left\{ \left(\frac{a_2-B/A}{2}, \frac{-4}{a_2} , a_2\right)\ |\ a_2 \in \QQ \right \}  \cap\calo^3\right ). 
 \end{align*}

In particular, for the third curve to have points in $\calo^3$ it is necessary that $B/A \in \calo$.
Therefore we can write

\begin{equation*}
     V(F)_{1,2}(\calo) = \begin{cases}
          R_1(\calo) &\qquad \text{if }B/A \not\in \calo, \\
           R_1(\calo) \cup
C_1&\qquad \text{if }B/A\in\calo, B/2A\not \in\calo, \\
 R_1(\calo) \cup R_2(\calo) \cup C_2&\qquad \text{if }B/2A\in\calo. 
     \end{cases}
\end{equation*}
where \begin{align*}
    C_1&= \left\{ \left(\frac{\pm p^u-B/A}{2}, \mp 4 p^{-u}, \pm p^u \right)\ |\ u \in \ZZ\right \}\\
    \intertext{and}
     C_2&=\left\{ \left(\frac{\pm x p^u -B/A}{2}, \frac{\mp 4 p^{-u}}{x} , \pm x p^u \right)\ |\ x\in\{2,4\}, u\in \ZZ \right \}.
\end{align*}






\item \textbf{If $\bm{A\neq 0}$ and $\bm{B^2-4AC \neq 0}$.} 
If $a_2=0$, equation~\eqref{eq:1-2bis} shows that also $a_1=0$. We will then assume $a_2 \neq 0$.  From the expression of $\mathrm{Quad}(F)$ in Table~\ref{tab:polynomials} we see that $a_2 \neq 0$ implies $F(b_1)\neq 0$. Thus, solving system \eqref{eq:1-2} in the variable $b_1$, we find 
\begin{equation}\label{eq:a1a2}
    a_1= -\frac{F'(b_1)} {F(b_1)}, \quad\quad a_2= \frac {F'(b_1)} {A}.
\end{equation}
In conclusion,
\begin{align*}
 V(F)_{1,2}(\calo)&=R_1(\calo) \cup \left ( \left \{ \left  (b_1, -\frac{F'(b_1)} {F(b_1)} ,  \frac {F'(b_1)} {A}\right )\ |\ b_1 \in \calo \mbox{ and } F(b_1)\neq 0  \right \}\cap\calo^3 \right ) . 
 \end{align*}
 Also in this case, if $B/A\not\in \calo$ then the second curve does not contain points in~$\calo^3$.
\end{itemize}



\begin{remark}
\label{rem:siegel12}
 $V(F)_{1,2}$ always has a component containing the line~$R_1$. If $A(B^2-4AC)\neq 0$, then  $V(F)_{1,2}$ has exactly one additional component, which is a curve of genus 0 and 3 points at infinity, therefore it has at most finitely many $\calo$-points by Siegel's theorem~\cite[Rem.\,D.9.2.2]{hindrySilverman2000}.
If $B^2=4AC\not =0$, there are exactly two additional components: one is the line $a_2=0,b_1=A/2$ and the other is the rational curve with two points at infinity given by $a_1a_2=-4, b_1=\frac{A+a_2}{2}$.
\end{remark}

We have seen that a purely periodic  continued fraction $[\overline{a_1,a_2}]$ of type $(0,2)$ converges if and only if $|a_1a_2|_p>1$. Since the convergence of a periodic continued fraction only depends on its periodic part, we deduce the following:

\begin{theorem}\label{teo:cond1-2}
Assume 
$b_1\in\calo$ and $F(b_1)\neq 0$. Then, 
there exists a point $(b_1,a_1,a_2)\in V(F)^{\mathrm{con}}_{1,2}$ if and only if the following conditions hold:
\begin{align}\label{eq:condizioni1-2}
\begin{cases}
    \frac{F'(b_1)}{F(b_1)} \in\calo,\\
B/A \in\calo, \\
|F'(b_1)|^2_p>|A|_p|F(b_1)|_p.
\end{cases}
\end{align}
In this case $a_1$ and $a_2$ are as in~\eqref{eq:a1a2}.
\end{theorem}
The third condition in~\eqref{eq:condizioni1-2} implies that $|b_1|_p\leq\max( |B/A|_p,|C/A|^{1/2}_p)$, since, otherwise, both the left-hand side and right-hand side of the inequality would be equal to $|A|^2_p|b_1|_p^2$. We remark that, by Proposition~\ref{prop:02-Reducibility}, conditions \eqref{eq:condizioni1-2} can be satisfied even when the polynomial $F$ (and therefore the polynomial associated to the purely periodic part of the corresponding continued fraction) is reducible. 
\begin{corollary}
\label{cor:Vcon12}
    $V(F)^{\mathrm{con}}_{1,2}$ is finite.
\end{corollary}

\begin{proof}
    Theorem~\ref{teo:cond1-2} and the discussion at the beginning of this section show that convergent points exist only when $A(B^2-4AC)\neq 0$. The thesis then follows from Remark~\ref{rem:siegel12}.
\end{proof}
Finally, we remark that the involution defined in~\eqref{eqn:inv} preserves convergence. Namely, if conditions \eqref{eq:condizioni1-2} are satisfied for $b_1$, then they are also satisfied for $-b_1-B/A$, so that 
\begin{align}
\label{eqn:CF12}
\left(b_1, -\frac{F'(b_1)} {F(b_1)} ,  \frac {F'(b_1)} {A}\right ),\left(-b_1-B/A,  \frac{F'(b_1)} {F(b_1)} , - \frac {F'(b_1)} {A}\right )\in V(F)_{1,2}^{\rm{con}}. 
\end{align}
It is easy to see that the two corresponding continued fractions converge respectively to the two different roots of $F$.
 \begin{examples}\label{examples12} \ 
\begin{itemize}
    \item[a)] 
Assume that $F(x)=x^2+C$ with $C= -b_1^2\pm 2^ip^k$ where $b_1\in\calo$, $i\in \{0,1\}$ and $k> 2v_p(b_1)$. Then, conditions \eqref{eq:condizioni1-2} are satisfied and, by \eqref{eqn:CF12},
\begin{equation*}
    \left(b_1,-\frac{2b_1}{\pm 2^ip^k},2b_1 \right) \in V(F)^{\rm con}_{1,2}.
\end{equation*}
We point out that, in the very special case when $b_1 \in \ZZ$ with $1\le |2b_1|_{\infty}\le \frac{p-1}{2}$ and $C=-b_1^2+p$, the expansion
$$  \left(b_1,-\frac{2b_1}{p},2b_1 \right) \in V(F)^{\rm con}_{1,2} $$
is the one given by the Browkin II algorithm \cite{Browkin2000}, as shown by \cite[Thm.\,8]{BarberoCerrutiMurru2021}.

 \item[b)] Assume now that $F(x)=x^2+C$ with $C\in\ZZ$ such that $p\nmid C$. Then we see that $V(F)^{\rm con}_{1,2}\not=\varnothing$ if and only if there exist $b\in\ZZ$ such that $|b|_p=1$ and $c\in\ZZ$ a divisor of $\gcd(C,b)$ such that $b^2+C=2^icp^k$ with $k\geq 1$ and $i\in\{0,1\}$. In this case, \begin{equation*}
    \left(b,-\frac{2b}{ 2^icp^k},2b\right)\in V(F)^{\rm con}_{1,2}.
\end{equation*}
Notice that this implies
\begin{equation*}
    \left(\frac b c,-\frac{2b}{ 2^ip^k},2\frac b c\right)\in V\left (G\right )^{\rm con}_{1,2},
\end{equation*}
where $G(x)=x^2+\frac{C}{c^2}$.
 \item[c)] As an example of polynomial $F$ having roots in $\QQ_p$ and such that $V(F)_{1,2}^{\rm con}$ is empty, we notice that $V(F)_{1,2}^{\rm con}=\varnothing$ with $F(x)=x^2+5$ for every $p\equiv 1\pmod {35}$. 
 We need to show that $b^2=2^i c p^r+5$ has no solutions with $b\in\ZZ$, $i\in\{0,1\}$, $k\geq 1$ and $c\in\{1,5\}$. If $c=5$, this gives $\frac {b^2} 5= 2^ip^k+1$, that is $0\equiv 2^i+1\equiv \pm 2\pmod 5$, a contradiction. If $c=1$ and $i=0$, this gives a contradiction modulo~7. If $c=1$ and $i=1$, this gives a contradiction modulo 5. This example can be generalized to any polynomial $F(x)=x^2+C$ with $C$ prime such that none of $C+1,C+2$ is a square and $p$ is in some arithmetic progression.
 \item[d)] It is possible to find a polynomial $F$ such that $\#V(F)_{1,2}^{\rm con}>2$. Indeed, we have that
  $[7,\overline{\frac{2}{27},2}]$, $[5,\overline{-\frac{2}{27},-2}]$ and $[11,\overline{\frac{10}{3},10}]$, $[1,\overline{-\frac{10}{3},-10}]$ all converge to $3$-adic solutions of $F(x)=x^2-12x+8$, i.e.\ $6\pm 2\sqrt{7}$.
  \item[e)] For a polynomial $F(x)=x^2+C$, through the identity
\[
2p^{r+s}-\left(\frac{2p^r+p^s-1}{2}\right)^2=2p^{r}-\left(\frac{2p^r-p^s+1}{2}\right)^2
\]
and point (a) one can find examples of $C$ such that $|V(F)_{1,2}^{\rm con}|>2$.
  For example, $[10,\overline{-\frac{10}{27},20}]$, $[8,\overline{-\frac{8}{9},16}]$ and $[-10,\overline{\frac{10}{27},-20}]$, $[-8,\overline{\frac{8}{9},-16}]$ all converge $3$-adically to a square root of $46$.
\end{itemize}
\end{examples}

\subsection{Type \texorpdfstring{$(0,3)$}{(0,3)}}
\label{sect:03}
The equations of $V(F)_{0,3}$ are
\begin{equation}\label{eq:0-3}\begin{cases}  Aa_1a_2a_3 + Ba_2a_3 + Aa_1 - Aa_2 + Aa_3 + B  &=0,\\  Aa_1a_2 + Ca_2a_3 + A + C&=0,\\
Ca_1a_2a_3 - Ba_1a_2 + Ca_1 - Ca_2 + Ca_3 - B&=0.
\end{cases}
\end{equation} 
Eliminating $a_2$ from the equations we get
\begin{equation}
    a_2=-\frac{A+C}{Aa_1+Ca_3}
\end{equation}
and
\begin{equation}
\label{eqn:elimin03}
    Aa_1^2+Ca_3^2+Ba_1-Ba_3+A+C=0.
\end{equation}
\begin{remark}
The discriminant of~\eqref{eqn:elimin03}, seen as a quadratic polynomial in $a_1$, is
   \[\Delta(a_3^2+1)-(Ba_3-2A)^2.
   \]
   We deduce from this that if $\Delta<0$ then $V(F)_{0,3}(\RR)=\varnothing$.
\end{remark}

We study $V(F)_{0,3}(\calo)$ by considering different cases.
\begin{itemize}
\item If $A=0$ and $B=0$, we have that $V (F )_{0,3}(\mathcal O)= \varnothing$; indeed, combining the second and third equation of \eqref{eq:0-3} we get $a_2^2=-1$, which is impossible in $\mathcal O$;
\item similarly, if $B=C=0$ then $V (F )_{0,3}(\mathcal O)= \varnothing$;
\item if $A=C=0$, the first equation of \eqref{eq:0-3} implies that $a_2a_3=-1$ and the third equation implies that $a_1a_2=-1$; therefore
$$ V(F)_{0,3}(\mathcal O)=\{ (a, -1/a, a)\ |\ a \in \mathcal O^*\}. $$

None of these expansions are convergent. Indeed, condition (ii) of Theorem~\ref{criterio.di.convergenza.padica} is never satisfied because, for $j=1$ it implies that $v_p(-1/a)>0$, but for $j=3$ it would also imply that $v_p(a)>0$.

\item if $A=0$ and $B, C \neq 0$, from \eqref{eq:0-3} we have 
\begin{equation}\label{eq:0-3_bis}\begin{cases}
a_2a_3&=-1, \\
a_1&=-\frac{C}{B}(a_3^2+1)+a_3,
\end{cases}
\end{equation}
which implies in particular that $\frac{C}{B}(a_3^2+1) \in \mathcal O$ and $a_3 \in \mathcal O^*$; we write $a_3=\pm p^\ell$ for some $\ell\in\ZZ$. We can assume $B$ and $C$ coprime, which implies that $B$ divides $(a_3^2+1)$ in $\mathcal O$. If we write $B=p^k B_1$ with $p \nmid B_1$ and we consider the order of $p$ in $(\ZZ/B_1 \ZZ)^*$, the last condition can be satisfied if and only if such order is a multiple of $4$, say $4s$. In this case, $V(F)_{0,3}(\mathcal O)$ contains infinitely many points of the form 
\[\left (-C/B(p^{2s(1+4t)}+1)\pm p^{s(1+4t)}, \mp p^{-s(1+4t)}, \pm p^{s(1+4t)} \right ) \mbox{ with $t \in \ZZ$}.\]

Studying convergence through Theorem~\ref{criterio.di.convergenza.padica} we see that 
$$V(F)_{0,3}^{\mathrm{con}}=\{(a_1,a_2,a_3)\in V(F)_{0,3}(\mathcal O) \mid v_p(a_3)<0 \text{ and }a_3\neq B/C\}.$$
In fact, using \eqref{eq:0-3_bis} to express $a_1,a_2$ in terms of $a_3$, we see that condition (ii) of Theorem~\ref{criterio.di.convergenza.padica} with $j=1$ implies that $v_p(a_3)<0$, which in turn implies that condition (i) is satisfied; condition (ii) with $j=2$ implies that $a_3\neq B/C$ (since otherwise we should have $v_p(a_3)>0$), and condition (ii) with $j=3$ is always satisfied because $C\neq 0$.



\item if $C=0,A,B\neq 0$, we recall that, by Proposition \ref{prop:scambio}, we have that $V(Ax^2+Bx)_{0,3}$ is isomorphic to $V(-Bx+A)_{0,3}$ through the isomorphism $(a_1,a_2,a_3)\mapsto (a_3,a_2,a_1)$. Using the previous case, we have that $V (F )_{0,3}(\mathcal O)$ is nonempty if and only if the order of $p$ in $(\ZZ/B_1 \ZZ)^*$ is a multiple of $4$, where $B=p^k B_1$ and $p \nmid B_1$; we denote it by $4s$. In this case, $V(F)_{0,3}(\mathcal O)$ contains infinitely many points of the form 
\[\left (   \pm p^{s(1+4t)}      , \mp p^{-s(1+4t)}, A/B(p^{2s(1+4t)}+1)\pm p^{s(1+4t)} \right ) \mbox{ with $t \in \ZZ$}.\]
We first recall that, as seen in Remark \ref{rmk:convnotpreserved}, the previous isomorphism does not preserve the convergence, which has to be studied separately. 
From \eqref{eq:0-3} we have 
\begin{equation}\label{eq:0-3_bis2}\begin{cases}
a_1a_2&=-1, \\
a_3&=\frac{A}{B}(a_1^2+1)+a_1.
\end{cases}
\end{equation}
Using Theorem~\ref{criterio.di.convergenza.padica} we see that 
$$V(F)_{0,3}^{\mathrm{con}}=\{(a_1,a_2,a_3)\in V(F)_{0,3}(\mathcal O) \mid v_p(a_1)>0 \text{ and }a_1\neq -B/A\}.$$
In fact, using \eqref{eq:0-3_bis2} to express $a_2,a_3$ in terms of $a_1$, we see that condition (ii) of Theorem~\ref{criterio.di.convergenza.padica} with $j=1$ is always satisfied because $A\neq 0$; condition (ii) with $j=3$ implies that $v_p(a_1)>0$, 
 which in turn implies that condition (i) is satisfied; finally, condition (ii) with $j=2$ implies that $a_1\neq -B/A$ (since otherwise we should have $v_p(a_3)>0$ and so $v_p(a_1)<0$).

\end{itemize}


In the general case, describing $V(F)_{0,3}$ in terms of $A,B,C$---as we did, for example, for type $(0,2)$---turns out to be too involved. We will only focus on the significant case of pure radicals, namely $A=1, B=0$ and $C=-d$ for some integer $d$.


\subsubsection{Pure integer radicals}\label{sssect:pureintegerradicals}
Equation~\eqref{eq:0-3}, with $A=1,B=0,C=-d$ gives the system of equations
\begin{equation}\label{eq:0-3tracciazero}
  \left\{\begin{array}{@{}l@{}}
    a_2-a_1-a_3-a_1a_2a_3=0,\\
    d(a_2a_3+1)=a_1a_2+1.
  \end{array}\right. 
\end{equation}

\begin{proposition}\label{prop:soluzioni0-3}
    Assume that $d\neq 0,1$.
    The variety $V(F)_{0,3}$ is an irreducible smooth rational curve in $\AA^3$ with three points at infinity, and the set $V(F)_{0,3}(\calo)$ is at most finite. 
    
    All converging $p$-adic expansions of $\sqrt{d}$ of type $(0,3)$ are of the form
\[
\left[\overline{a_1,\frac{v}{p^s},a_3}\right],
\]
where  $a_1,a_3,v,s\in\ZZ$, $s\geq 1$, $p\nmid a_1a_3v$, $v \mid d-1$  and
\begin{align}
& a_1^2-da_3^2 =d-1,  \label{eqn:pell} \\
& a_1-da_3 = \frac {d-1} v p^s.\label{eq:diciannove}
\end{align}
Furthermore, every such quadruple $a_1,a_3,v,s$ gives rise to a converging $p$-adic expansion of $\sqrt{d}$ of type $(0,3)$ for some $d \in \ZZ$.
\end{proposition}

\begin{proof}
Eliminating $a_2$ from \eqref{eq:0-3tracciazero}, we get \eqref{eqn:pell}.
The affine curve $V(F)_{0,3}$  defined by \eqref{eq:0-3tracciazero} is irreducible for $d\neq 0,1$. Embedding $\AA^3$ in $\PP^3$ with coordinates $(a_1:a_2:a_3:z)$ and assuming $d\neq 0,1$ we see (using also equation~\eqref{eqn:pell}) that the projective completion of $V(F)_{0,3}$ has the three points at infinity $(0:1:0:0),(\pm\sqrt{d}:0:1:0)$.

Therefore, Siegel's theorem on integral points~\cite[Rem.\,D.9.2.2]{hindrySilverman2000} guarantees that $V(F)_{0,3}(\calo)$ is at most finite.

Assume now that $(a_1,a_2,a_3)\in V(F)_{0,3}^{\mathrm{con}}(\calo)$.
From \eqref{eqn:pell} we see
\begin{equation*}
    (a_1+a_3)(a_1-da_3)=(d-1)(1-a_1a_3).
\end{equation*}
Moreover, $1-a_1a_3$ and $a_1-da_3$ are non-zero---otherwise the previous equation would imply $a_1=\pm i\not\in \calo$ or $(a_1,a_3)=(\pm \sqrt{d},\pm 1/\sqrt{d})$, which is incompatible with the second equation of~\eqref{eq:0-3tracciazero}.

From \eqref{eq:0-3tracciazero} we can then write
\begin{equation}
\label{eqn:a2Type(0,3)}
    a_2=\frac {a_1+a_3}{1-a_1a_3}= \frac {d-1} {a_1-da_3}.
\end{equation}
The quantity $a_1a_2a_3+a_1+a_2+a_3$ is equal to $2a_2$, which is also equal to $2\frac{a_1+a_3}{1-a_1a_3}$; the $p$-adic convergence of $[\overline{a_1,a_2,a_3}]$ implies by Theorem~\ref{criterio.di.convergenza.padica} that $v_p(a_2)<0$.
If one between $a_1$ or $a_3$ had negative $p$-adic valuation, so would the other by \eqref{eqn:pell} and the two valuations would be equal (we recall that $p\neq 2$); but then $v_p(a_2)=v_p(\frac{a_1+a_3}{1-a_1a_3})>0$ against our assumption. Therefore, $v_p(a_1),v_p(a_3)\geq 0$, i.e. $a_1,a_3\in\ZZ$.

If at least one between $v_p(a_1),v_p(a_3)$ were positive, then $v_p(\frac{a_1+a_3}{1-a_1a_3})\geq0$ which is not the case; therefore $v_p(a_1)=v_p(a_3)=0$.

Finally, since $a_2\in\calo$ we must have $a_1-da_3=wp^t$ for some integer $w$ dividing $d-1$ and prime to $p$, and $t>v_p(d-1)$. Let  $v$ be the  prime-to-$p$ part of $\frac{d-1} w$, and $s=t-v_p(d-1)$. Then, we get $a_2=\frac v {p^s}$ and this proves the first part of the statement.

Now take a quadruple $(a_1,a_3,v,s)$ satisfying the hypotheses and define $a_2=v/p^s$.
Notice that, for such a quadruple, $a_1 a_3\neq -1$, because $a_1,a_3$ would need to be equal to $\pm 1$, which is not compatible with \eqref{eqn:pell}.

Part (i) of Theorem~\ref{criterio.di.convergenza.padica} is verified. Part (ii) with $j=1$ or $j=3$ is always satisfied because $v_p(a_1 a_2),v_p(a_2 a_3)<0$. Part (ii) with $j=2$ is also satisfied because $a_1 a_3+1\neq 0$. This proves the theorem.
\end{proof}
The above result is in agreement with the sufficient condition given in~\cite[Cor.\,10]{Murru2023}.
\begin{remark}
\label{rem:empty03}
From \eqref{eq:0-3tracciazero} it is easy to see that, if $d=0$, the variety $V(F)_{0,3}$ is the union of lines of the form $(i,i,a_3)$ and $(-i,-i,a_3)$, where $i^2=-1$, so it is reducible over $\overline{\QQ}$. Notice that for $d=0$ the set $V(F)_{0,3}(\mathcal O)$ is empty.

For $d=1$, the variety $V(F)_{0,3}$ is the union of lines of the form $(a_1,0,-a_1)$ and of the curve parametrized by $(a_1,\frac{2a_1}{1-a_1^2},a_1)$, so it is reducible over $\QQ$. In this case the lines and the curve intersect in $(0,0,0)$ which is a singular point of $V(F)_{0,3}$ and $V(F)_{0,3}(\mathcal O)\neq \varnothing$.

Finally, for $d<0$, equation~\eqref{eqn:pell} shows that  $V(F)_{0,3}(\calo)$ is empty.
\end{remark}

\begin{corollary}
\label{thm:dCondition}
Suppose that $d\in\ZZ$ is divided by $4$ or by any prime congruent to $-1$ modulo $4$ or $d\leq 0$. Then, for every prime $p$,  $\sqrt{d}$ does not have a $p$-adic expansion of type $(0,3)$.
\end{corollary}


\begin{proof}
If $\sqrt{d}$ has a $p$-adic expansion of type $(0,3)$, then equation~\eqref{eqn:pell} modulo $4$ (resp.\ any prime divisor $q$ of $d$) implies that $d\not \equiv 0 \pmod 4$ (resp.\ $q\not \equiv -1 \pmod 4$).
\end{proof}

\begin{remark}
Notice that the converging $p$-adic expansions of $\sqrt{d}$ of type $(0,3)$ given by Proposition \ref{prop:soluzioni0-3} have $a_1, a_3 \in \ZZ$, so they are not expansions of Browkin or Ruban type (where the partial quotients are elements of $\calo$ with negative $p$-adic valuation). More specifically, these expansions do not arise from any algorithm involving the use of a single `floor function' (for the formal definition of floor functions see \cite{CapuanoMurruTerracini2022}). In particular, this is again consistent with Bedocchi's result on Browkin continued fractions already discussed in Remarks~\ref{rem:bedocchi01} and~\ref{rem:bedocchi02}.
\end{remark}

\begin{proposition}
\label{prop:francescoBound}
{Let $d$ be a positive square-free integer, and suppose that $\sqrt{d}$ has a converging $p$-adic expansion
\(
\left[\overline{a_1,\frac{v}{p^s},a_3}\right]
\) as in Proposition~\ref{prop:soluzioni0-3}.} Let $k$ be the greatest common divisor of $a_1+a_3$ and $1-a_1a_3$. Then \begin{equation}\label{eqn:FrancescoBound}p^s\leq  k\frac{ (d+1)^{3/2}+2d}{\left| d+1 \right|^2_p} .\end{equation}
\end{proposition}

\begin{proof}
By~\eqref{eqn:a2Type(0,3)} we have $1-a_1 a_3=k p^s$. This equality can be used to eliminate $a_3$ from~\eqref{eqn:pell}, yielding

\begin{equation}
\label{eqn:quartic_a1}
    a_1^4-(d-1)a_1^2-d(kp^s-1)^2=0.
\end{equation}
Viewing it as a quadratic equation in $a_1^2$, we have that the discriminant must be a square, hence there exists $u \in \ZZ$ such that
\begin{equation}\label{eq:u}
(d-1)^2+4d(kp^s-1)^2=u^2.
\end{equation}
Choose $u$ to be positive (it cannot be $0$ since the left-hand side of~\eqref{eq:u} is always $\geq 1$).
Moreover,
\begin{align}
        4dkp^s(kp^s-2) &= 4d(kp^s-1)^2-4d \nonumber\\
    &=u^2-(d-1)^2-4d \nonumber\\
    &=u^2-(d+1)^2 \nonumber\\
   &=(u+d+1)(u-d-1) \label{Ex4.11Eqn25}.
\end{align}
From this equality we obtain an effective upper bound for $s$ in terms of $p,d$, and $k$.
The idea is that the two factors in \eqref{Ex4.11Eqn25} are almost coprime, so that the factor $p^s$ must go almost entirely in one of them, say for now in the first one. Then we can write
\begin{equation}\label{Ex4.11Eqn26}
  \left\{\begin{array}{@{}l@{}}
    u+d+1=2\dfrac{ag}{p^x} p^s,\\[3ex]
    u-d-1=2\dfrac{d k p^x}{ag}(kp^s-2),
  \end{array}\right.
\end{equation}
where $a$ is a positive divisor of $kd$, $x$ is an integer between 0 and $v_p(\gcd(u,d+1))$ and $g$ is a positive divisor of $kp^s-2$.
These two factors should have roughly the same size. Taking the difference we get
\begin{equation*}
    d+1=\frac{a^2 g^2 - d k^2 p^{2x}}{agp^x}p^s + \frac{2dk p^x}{ag}.
\end{equation*}
The coefficient in front of $p^s$ cannot be 0 because $d$ is not a square, so we obtain 
\begin{equation*}
    p^s=\left(d+1-\frac{2dk p^x}{ag}\right)\frac{agp^x}{a^2 g^2 - dk^2p^{2x}}\leq \left|d+1-\frac{2d kp^x}{ag}\right|agp^x\leq (d+1)agp^x + 2d kp^{2x}.
\end{equation*}
We are left with the task of bounding $g$ independently of $s$.
To do so, take the quotient of the two equations \eqref{Ex4.11Eqn26}
\begin{align}
    \frac{u-d-1}{u+d+1}&=dk\left(\frac{p^x}{ag}\right)^2\left(k-\frac{2}{p^s}\right) \nonumber\\ \label{eqn:ratio}
    \left(\frac{ag}{p^x}\right)^2&=dk\left(k-\frac{2}{p^s}\right)\frac{u+d+1}{u-d-1} < d k^2 \left(1+\frac{2(d+1)}{u-d-1}\right).
\end{align}

The bound in~\eqref{eqn:FrancescoBound} is larger than $\frac{k(d+1)^2}{\sqrt{d}}+1$, therefore we can now assume freely that $p^s\geq \frac{k(d+1)^2}{\sqrt{d}}+1$, so that, from~\eqref{eq:u}, $u>2\sqrt{d}(kp^s-1)\geq 2k^2(d+1)^2$. Therefore we get 
\begin{align*}
\left(\frac{ag}{p^x}\right)^2 & < d k^2 \left(1+\frac{2(d+1)}{u-d-1}\right) < d k^2 \left(1+\frac{2(d+1)}{2k^2(d+1)^2-d-1}\right)\\
&= d k^2 \left(1+\frac{2}{2k^2(d+1)-1}\right)< d k^2 \left(1+\frac{1}{dk^2}\right)= dk^2+1.
\end{align*}
Thus we obtain
\begin{align*}
   p^s&\leq (d+1)agp^x + 2d kp^{2x}=p^{2x}\left( (d+1)\frac{ag}{p^x}+2kd \right)< p^{2x}\left( (d+1)\sqrt{dk^2+1}+2kd \right)\\
    &\leq p^{2x}k\left( (d+1)\sqrt{d+1}+2d \right)\leq  k\frac{ (d+1)^{3/2}+2d}{\left| d+1 \right|^2_p}.
\end{align*}

Arguing with the role of the two factors of \eqref{Ex4.11Eqn26} reversed leads to the same bound.
\end{proof}
Setting $k=1$ in the statement of Proposition~\ref{prop:francescoBound} shows that, for fixed $d$, there is only a finite number of $p$-adic expansions of $\sqrt{d}$ of type $(0,3)$ of the form
\[
\left[\overline{a_1,\frac{a_1+a_3}{p^s},a_3}\right].
\]
 For example, for $k=1$ and $3\leq |d|\leq 10$ one can check by hand---for the primes allowed by Theorem~\ref{thm:dCondition} and equations~\eqref{eqn:FrancescoBound} and~\eqref{eqn:quartic_a1}---that only $\sqrt{10}$ has $p$-adic expansions of type $(0,3)$, which are
\[\left[\overline{ 13, \frac{9}{53}, -4}\right]_{53}\qquad \text{and} \qquad \left[\overline{ -13, -\frac{9}{53}, 4}\right]_{53}.\]
We remark that, for $k\neq 1$, one can find other $p$-adic expansions of $\sqrt{2}, \sqrt{5}$ and $\sqrt{10}$ of type $(0,3)$:
\begin{align*}
     \sqrt{10}&=\left[\overline{7,-\frac{9}{13},2}\right]_{13},  \tag{for $k=-1$}\\
     \sqrt{5}&=\left[\overline{7,\frac{2}{11},-3}\right]_{11}, \tag{for $k=2$}\\
     \sqrt{2}&=\left[\overline{17,\frac{1}{41},-12}\right]_{41}, \tag{for $k=5$}\\
    \sqrt{10}&=\left[\overline{-57,\frac{3}{41},-18}\right]_{41}, \tag{for $k=-25$}\\
    \sqrt{10}&=\left[\overline{253,-\frac{9}{547},80}\right]_{547}, \tag{for $k=-37$}\\
    \sqrt{10}&=\left[\overline{487,\frac{9}{2027},-154}\right]_{2027}, \tag{for $k=37$}\\
    \sqrt{10}&=\left[\overline{2163,-\frac{3}{1559},684}\right]_{1559} .\tag{for $k=-949$}
\end{align*}

\subsubsection{Finiteness results}
\label{sec:finiteness03}
In light of equation~\eqref{eqn:pell} in Proposition~\ref{prop:soluzioni0-3}, the structure of the set $V(F)_{0,3}^{\mathrm{con}}(\calo)$ is related to generalized Pell equations, which are themselves related to linear recurrence sequences. In this section we will recall some known effective bounds on linear recurrence sequences in order to derive the finiteness of some continued fraction expansions of type $(0,3)$ representing roots of integers.

The generalized Pell equation, i.e.\ the equation of the form $x^2-dy^2=n$, has been widely studied. The solutions correspond to elements of norm~$n$ in $\mathbb{Z}[\sqrt{d}]$. When $n=1$, solutions form a cyclic group $\mathcal{U}_d$ under the Brahmagupta product. For general $n$, the set of solutions is equipped with
 an action of $\mathcal{U}_d$. As in \cite[Ch.\,6,~\S 58]{Nagell1951},
we call \emph{classes of solutions} the orbits of such action. We shall say that  a solution $(\x,\y)$ of $x^2-dy^2=n$
is the \emph{fundamental solution}
in its class if $\y$ is the minimal non-negative among all
solutions in the class. If two solutions in the class have
the same minimal non-negative $\y$, then the solution with
$\x > 0$ is the fundamental solution.

 If $d$ is not a perfect square, then there are at most finitely many classes of solutions \cite[Thm.\,109]{Nagell1951}. Moreover, the fundamental solution $(\x,\y)$ is explicitly bounded. In the case $n\geq 0$, it lies in the rectangle
 \begin{align*}
           & 0< |\x|\leq \sqrt{\frac 1 2 (u^*+1)n}, \\
    & 0\leq  \y\leq  \frac {v^*}{\sqrt{2(u^*+1)}}\cdot\sqrt{n},
 \end{align*}
 where $(\x^*,\y^*)$ is the  positive generator of $\mathcal{U}_d$.\par
From now on we will deal with the case $n=d-1$. Starting from a fundamental solution (if any) $(\x_0,\y_0)$ of $x^2-dy^2=d-1$, we can construct two sequences $(u_i,v_i)$ and $(u_{-i},v_{-i})$ obtained respectively via multiplication (Brahmagupta product) by  $(\x^*,\y^*)$  or by its inverse $(\x^*,-\y^*)$. In other words,
if
$$M=\begin{pmatrix} \x^* & d\y^*\\ \y^* & \x^*\end{pmatrix}, 
$$
then, for every $i\in\ZZ$,
$$\begin{pmatrix} u_i\\ v_i\end{pmatrix}=M^i \begin{pmatrix} u_0\\ v_0\end{pmatrix}.$$
Notice that the sequences $(u_i)$, $(v_i)$, $(u_{-i})$, $(v_{-i})$ satisfy the second order recurrence
\begin{equation}
    \label{eqn:second_order_recurrence}
    x_{i+2}=2\x^* x_{i+1}-x_i,
\end{equation}
so that the same recurrence relation is satisfied also by the sequences $(\x_i-d\y_i)$ and $(\x_{-i}-d\y_{-i})$.  We recall the following result concerning the occurrence of powers of integers in linear recursive sequences.
\begin{theorem}[\cite{Petho1982}]\label{teo:Petho}
Let $(x_n)$ be a sequence of integers satisfying the binary linear recurrence relation
 $$x_{n}=hx_{n-1}-kx_{n-2},$$
 with $h,k\in\ZZ$ such that $h\not =0$, $(h,k)=1$ and $h^2\not=ik$ for $i\in\{1,2,3,4\}$.
 Assume that the initial values $x_0,x_1$ are not both zero and that $h^2-4k$ is not a perfect square if the quantity $k(x_1^2-hx_1 x_2+kx_0^2)$ is zero.
 
 Let $S$ be a finite set of prime numbers and consider the diophantine equation
 \begin{equation}\label{eq:diophpowers}
 x_n=wz^e, 
 \end{equation}
 with $n,e\in \NN,\ z\in \ZZ$,  $e\geq 2$, and $w$ an integer whose prime divisors all lie in $S$.
Then, there are effectively computable constants $C_1,C_2,C_3$ depending only on $h,k,S,x_0,x_1$ such that every solution as above of  equation \eqref{eq:diophpowers} satisfies
\begin{enumerate}[label=(\roman*)]
    \item \label{itm:z>1} $\max\{|w|,|z|,n,e\} < C_1, \quad $ if $|z|>1$;
    \item \label{itm:z1} $\max\{|w|,n\} < C_2, \quad $ if $|z|=1$;
    \item $n < C_3, \quad $ if $z=0$.
\end{enumerate}
\end{theorem}


\begin{theorem}
\label{thm:petho2}
Let $d$ be a square-free positive integer.
\begin{enumerate}[label=(\alph*)]
    \item \label{itm:petho_a} For every $p$, there are finitely many periodic $p$-adic continued fractions of type $(0,3)$ converging to $\sqrt{d}$. 
    Moreover, the maximum exponent $s$ and index $n$ are effectively computable.
    \item \label{itm:petho_b} There are finitely many $p$ and periodic $p$-adic continued fractions of type $(0,3)$ converging to $\sqrt{d}$  and such that the exponent $s$ of $p$ in the denominator of $a_2$ is~$>1$.
\end{enumerate}
\end{theorem}
\begin{proof}
 From~\eqref{eqn:second_order_recurrence} it is easily seen that the hypotheses of Theorem \ref{teo:Petho} are satisfied for the sequences $x_n=\x_n-d \y_n$ (resp.\ $x_n=\x_{-n}-d \y_{-n}$) with $h=2\x^*$ and $k=1$.
   To prove~\textit{\ref{itm:petho_a}}, we take $S=\{\hbox{prime divisors of } d-1\}\cup \{p\}$ and $z=1$. A $p$-adic continued fraction \(
\left[\overline{a_1,a_2,a_3}\right]
\) converging to $\sqrt{d}$ has the form dictated by Proposition~\ref{prop:soluzioni0-3}. In particular, as $a_1,a_3$ satisfy~\eqref{eqn:pell}, $(a_1,a_3)$ must be one of the pairs $(\x_n,\y_n)$, which implies that $a_1-da_3$ is one of the $x_n$. Moreover, by the equality \eqref{eq:diciannove}, we see that $a_1-da_3$  satisfies~\eqref{eq:diophpowers} with $w=\frac{d-1}{v}p^s$. The claim then follows from Theorem~\ref{teo:Petho}\textit{\ref{itm:z1}}.

   To prove~\textit{\ref{itm:petho_b}}, we analogously apply Theorem \ref{teo:Petho}\textit{\ref{itm:z>1}} with $S=\{\hbox{prime divisors of }d-1\} $, $z=p$ and $e=s$.
\end{proof}

Although Theorem \ref{teo:Petho} reduces in principle the problem of finding all the solutions
of \eqref{eq:diophpowers} to a finite amount of computations, from a practical point of view the possibility
of using brute force is illusory since the computations involved are extremely heavy. 

\begin{example}\label{ex:4.24}\
    Consider the case $d=5$, and let $[\overline{a_1,a_2,a_3}]$ be a $p$-adic PCF  ($p\neq2$) converging to $\sqrt{5}$.  
Setting $p^s=|a_2|_p$, we show that $s$ must be equal to $1$. 
By~\eqref{eqn:pell} and~\eqref{eq:diciannove}, we are considering solutions of the equations $x^2-5y^2=4$ and $x-5y=wp^s$ for $w$ a divisor of 4.
The second equation tells us that $x\equiv 5y \pmod{p}$; putting this into the first equation gives $5y^2\equiv 1 \pmod{p}$. Therefore $\bigl( \frac{5}{p}\bigr)=1$ and, by reciprocity, $\bigl( \frac{p}{5}\bigr)=1$, so that $p\equiv \pm 1 \pmod{5}$.

The first equation also tells us that $x\equiv \pm 2 \pmod{5}$ and therefore from the second equation we obtain that $w=\pm 2$. Eliminating $x$ from the first equation we now find
\[
5y^2\pm 5 p^s y + p^{2s} -1=0.
\]
The discriminant of this equation as a quadratic equation in $y$ is $5(p^{2s}+4)$, therefore we must have that $p^{2s}+4=5u^2$, that is to say that $\frac{p^s+\sqrt{5}u}{2}$ has norm $-1$ in $\QQ(\sqrt{5})$. Let $\phi=\frac{1+\sqrt{5}}{2}$ be the fundamental unit in $\QQ(\sqrt{5})$.
It can be shown easily by induction that its powers can be written as $\phi^n=\frac{L_{|n|}+\mathrm{sg}(n)\sqrt{5}F_{|n|}}{2}$, where $L_n$ and $F_n$ are the $n$-th Lucas and Fibonacci numbers respectively. Therefore we are looking for prime powers in the sequence of Lucas numbers, but it is known from \cite[Thm.\,2]{MR2215137} that $1$ and $4$ are the only Lucas numbers to be perfect powers.

On the other hand, as we will see in Example~\ref{ex:pellEx}, PCFs of type $(0,3)$ converging to $\sqrt{5}$ do exist for (conjecturally infinitely) many primes.

\end{example}
\subsubsection{The case \texorpdfstring{$d=a^2+1$}{d=a^2+1}}
\label{sec:a2+1}
In this case, examples of PCFs can be found by looking for prime powers occurring in a suitable recursive sequence. From now on, we will assume $a>0$.
\begin{proposition}\label{prop:4.25}
Let $f_n(x)\in\ZZ[x]$ be the sequence of polynomials defined by the recurrence 
\begin{align*}
 &f_0(x)=1,  \\
 &f_1(x)=-2x^3+2x^2-2x+1,\\
 &f_{n+2}(x)=2(2x^2+1)f_{n+1}(x)-f_n(x).
 \end{align*}
    For every $n\in\NN$ such that $|f_n(a)|$ is a power of $p$ with exponent $\geq 1$,  there is a point $\left (a\x^*_n,\frac a {f_n(a)},a\y^*_n\right )\in V(x^2-d)_{0,3}^{\rm {con}}$, i.e.\ a $p$-adically  convergent continued fraction for $\sqrt{d}$.
\end{proposition}

\begin{proof}
    The fundamental unit of  $\QQ(\sqrt d)$ is $a+\sqrt{d}$ and has norm $-1$.  Therefore the fundamental solution of the Pell equation $x^2-dy^2=1$ is $(2a^2+1, 2a)$, and---since $a>0$---all positive solutions $(\x^*_n, \y^*_n)$ are obtainable by the recurrence formulae 
  $$\left\{
\begin{array}{ll}
\x^*_1 &=2a^2+1,\\
\y^*_1 &=2a,\\
\x^*_{n+1}&=\x^*_1\x^*_n + d\y^*_1\y^*_n,\\
\y^*_{n+1}&=\y^*_1\x^*_n+ \x^*_1\y^*_n.\end{array} \right . $$
A class of solutions of the Pell equation $x^2-dy^2=d-1$ is given by
$$\left\{
\begin{array}{ll}
\x_n&=a\x^*_n,\\
\y_n&=a\y^*_n.\end{array}\right .  $$
Then, $$\x_n-d\y_n=a(\x^*_n-d\y^*_n).$$ 
One can check that $f_n(a)=\x^*_n-d\y^*_n$. Therefore, we have 
$$\frac{d-1}{\x_n-d\y_n} = \frac{a}{\x^*_n-d\y^*_n}= \frac a {f_n(a)},$$
and the greatest common divisor between $a$ and $f_n(a)$ is $1$ because ${\x^*_n}^2-d{\y^*_n}^2=1$ implies $$(\x^*_n+\y^*_n)f_n(a)=(\x^*_n+\y^*_n)(\x^*_n-d\y^*_n)=1-(d-1)\x^*_n\y^*_n=1-a^2\x^*_n\y^*_n.$$
The thesis then follows from Proposition~\ref{prop:soluzioni0-3}.
\end{proof}

\begin{remark}
We point out that the recurrence formulae defining the polynomials $f_n(x)$ of the previous proposition resembles the recurrence defining the Chebyshev polynomials. It turns out that one can write $f_n(x)$ as a linear combination of the Chebyshev polynomials of first and second kind evaluated in $2x^2+1$ as follows:
\begin{equation*}
f_n(x)= T_n(2x^2+1)-2x(x^2+1)U_{n-1}(2x^2+1) \qquad \forall n \ge 1.
\end{equation*}
\end{remark}

\begin{example}\label{ex:pellEx}
 For $a=2$ looking for (negative) prime powers in the sequence $f_n(a)$ for $n\in\{1,\dots,10\}$ yields the following continued fraction expansions of $\sqrt{5}$:
\begin{align*}
  \left[\overline{18, -\frac{2}{11}, 8}\right]_{11} \tag{for $n=1$}\\
  \left[\overline{322, -\frac{2}{199}, 144}\right]_{199} \tag{for $n=2$}\\
  \left[\overline{5778, -\frac{2}{3571}, 2584}\right]_{3571} \tag{for $n=3$}\\
  \left[\overline{599074578, -\frac{2}{370248451}, 267914296}\right]_{370248451} \tag{for $n=7$}\\
  \left[\overline{10749957122, -\frac{2}{6643838879}, 4807526976}\right]_{6643838879} \tag{for $n=8$}\\
  \left[\overline{192900153618, -\frac{2}{119218851371}, 86267571272}\right]_{119218851371} \tag{for $n=9$}
\end{align*}

\end{example} 

Recall that the \emph{Bunyakowsky Conjecture} \cite{Bun57} asserts that, for every nonconstant irreducible polynomial $f(x)$ in $\ZZ[x]$ such that the values $f(n)$ are coprime for $n=1,\ldots$, there are infinitely many $n\in\ZZ$ such that $|f(n)|$ is a prime number.\\
Then, for instance, since $f_1(x)=-2x^3+2x^2-2x+1$ is irreducible, and it has constant term 1,  we expect to find infinitely many integers $d$ of the form $a^2+1$ and primes $p$ such $|f_1(a)|=p$, so that the $p$-adically convergent locus of $V(x^2-(a^2+1))_{0,3}$ contains the point $\left (a(2a^2+1),- \frac a  p, 2a^2\right) $.
The following conjecture has been tested for $n$ up to $10000$.
\begin{conjecture} \label{conj:irred} $f_n(x)$ is irreducible in $\ZZ[x]$ for every $n\in\NN$.\end{conjecture}

 If Conjecture \ref{conj:irred} and Bunyakowsky  Conjecture are true,
 then for every $n\in\NN$ there are infinitely many integers $a$ such that $|f_n(a)|$ is a prime $p$, giving rise to   points $\left (ax_n,-\frac a {p},ay_n\right )\in V(x^2-(a^2+1))_{0,3}^{\rm con}$.
\\
 The following conjecture, if true, would show that the condition on the exponent of $p$ in the second part of Theorem \ref{thm:petho2} cannot be dropped.
 \begin{conjecture}
For every positive integer $a$ there are infinitely many $n\in\NN$ such that $|f_n(a)|$ is a prime number.
     \end{conjecture}

\subsubsection{The negative Pell equation} \label{sss:4.6.4}
Assume now that the negative Pell equation $x^2-dy^2=-1$ admits a solution, and let $(s_1,t_1)$ be the fundamental solution for it. Then the positive solutions of $x^2-dy^2=1$ can be obtained by the recurrence formulae
  $$\left\{
\begin{array}{ll}
\x^*_1 &=s_1^2+dt_1^2,\\
\y^*_1 &=2s_1t_1,\\
\x^*_{n+1}&=\x^*_1\x^*_n + d\y^*_1\y^*_n,\\
\y^*_{n+1}&=\y^*_1\x^*_n+ \x^*_1\y^*_n,\end{array}\right.  $$
and the positive solutions of $x^2-dy^2=-1$ can be obtained by the recurrence formulae
  $$\left\{
\begin{array}{ll}
s_{n+1}&=\x^*_1s_n + d\y^*_1t_n, \\
t_{n+1}&=\y^*_1s_n+ \x^*_1t_n.\end{array}\right.  $$
A class of solutions of the generalized Pell equation $x^2-dy^2=d-1$ is given by $(u_n,v_n)$, where $u_n=s_n+dt_n$, $v_n=s_n+t_n$. We have  $\frac{d-1}{u_n-dv_n}=\frac{-1}{s_n}.$
Therefore, if $|s_n|$ is a prime number $p$,  the $p$-adically convergent locus of $V(x^2-d)_{0,3}$ contains the points $\left (\pm p+dt_n, \mp \frac 1 p, \pm p+t_n\right)$.
However, notice that $s_1$ divides $\y^*_1$, therefore it also divides $s_n$ for every $n\in\NN$. 
This implies that there is at most one prime in the sequence of $s_n$'s when $s_1\not=\pm 1$. Finally, $s_1=\pm 1$ can only occur if $d=2$: in this case we do find solutions for $p=7,41,239, 9369319, 63018038201, 489133282872437279, \dots$.

   \section{Some examples of type \texorpdfstring{$(1,3)$}{(1,3)}}
   \label{sec:1,3}
   The generic PCF of type (1,3) has the form $[b_1,\overline{a_1,a_2,a_3}]$.\\
   The equations of $V(F)_{1,3}$ are
\begin{equation*} \small{
    \begin{cases}
        Aa_1a_2a_3-2Aa_1a_2b_1-Ba_1a_2+Aa_1-Aa_2+Aa_3-2Ab_1-B&=0,\\
        -Aa_1a_2a_3b_1+Aa_1a_2b_1^2-Aa_1b_1-Aa_2a_3+Aa_2b_1-Aa_3b_1+Ab_1^2-Ca_1a_2-A-C&=0,\\
        \begin{aligned}[b]
        -Ba_1a_2a_3b_1+Ba_1a_2b_1^2-Ca_1a_2a_3+2Ca_1a_2b_1-Ba_1b_1-Ba_2a_3+Ba_2b_1+\\[-3pt]
        -Ba_3b_1+Bb_1^2-Ca_1+Ca_2-Ca_3+2Cb_1-B&=0.
        \end{aligned}
    \end{cases}}
\end{equation*}
   Eliminating $a_1$ from these equations one obtains
   \begin{equation}
       \label{eqn:elim13}
       Aa_1^2b_1^2+Ba_1^2b_1+2Aa_1b_1+Aa_3^2-2Aa_3b_1+Ab_1^2+Ca_1^2+Ba_1-Ba_3+Bb_1+A+C=0.
   \end{equation}
   \begin{remark}
        The discriminant of~\eqref{eqn:elim13}, seen as a quadratic polynomial in $b_1$, is
   \[
   \Delta(a_1^2+1)^2-4A^2(a_1a_3+1)^2.
   \]
   Thus, if $\Delta<0$ then $V(F)_{1,3}(\RR)$ is empty.
   \end{remark}

   Since the convergence only depends on the purely periodic part, Theorem~\ref{criterio.di.convergenza.padica} tells us that $[b_1,\overline{a_1,a_2,a_3}]$ converges $p$-adically if and only if $v_p(a_1a_2a_3+a_1+a_2+a_3)<0$.
    We focus on the case $F=x^2-d$. The equations of $V(F)_{1,3}$ are
\begin{equation*}
    \left\{ \begin{array}{ll}
        a_1a_2a_3+a_1+a_3-a_2=2b_1(a_1a_2+1),\\[0.5ex]
        a_3a_2+1=(d-b_1^2)(a_1a_2+1),
    \end{array} \right.
\end{equation*}
or, equivalently,

\begin{equation}\label{eq:rel5bis}
    \left\{ \begin{array}{ll}
        a_3=2b_1+\frac{ a_2-a_1}{a_1a_2+1},\\[0.5ex]
        (b_1^2-d)(a_1a_2+1)^2 + a_2^2+2b_1a_2(a_1a_2+1)+1=0.
    \end{array} \right.
\end{equation}
If we fix $b_1$, the second equation of~\eqref{eq:rel5bis} defines a conic in $a_2$ and $c=a_1a_2+1$.
\par In the remainder of this section, we assume that $d$ has the form $d=\a^2+1$ for some positive integer $\a$, and we study the integral points of the subvariety of $V(F)_{1,3}$ defined by $b_1=\a$. In this case, the second equation of~\eqref{eq:rel5bis} becomes
\begin{equation}\label{eq:rel6} -a_1^2a_2^2-2a_1a_2+a_2^2+2\a a_1a_2^2+2\a a_2=0.\end{equation}
For $a_2=0$ we obtain points 
$$[\a,\overline{ a_1,0,2\a-a_1}],$$
which are surely convergent if $v_p(\a),v_p(a_1)<0$.\\
We now find points with $a_2\not=0$.
Equation \eqref{eq:rel6} gives
$$a_2=\frac{2(\a-a_1)}{a_1^2-2\a a_1-1},$$
and the first equation in \eqref{eq:rel5bis} gives
$a_3=a_1$, 
so that we get a family of points
\begin{equation}
\label{eqn:cf31}
\left [\a,\overline{ a_1, \frac{2(\a-a_1)}{a_1^2-2\a a_1-1},a_1} \right ].\end{equation} 
A necessary condition for convergence in this case is
\begin{equation*}
    v_p\left(\frac{2 a_{1}^{2} \a + 4 a_{1} - 2 \a}{-a_{1}^{2} + 2 a_{1} \a + 1}\right)<0.
\end{equation*}
In particular, if $v_p(\a) \geq 0$ and $v_p(a_1)\neq 0$, \eqref{eqn:cf31} is not $p$-adically convergent.
\begin{example}
If $d=10$, $b_1=3,p=3$
    $$\sqrt{10}=\left [3,\overline{1,-\frac 2 3,1} \right ]_3.$$ The same holds in general for 
$$\sqrt{p^2+1}=\left [p,\overline{1,-\frac{p-1} p,1} \right ]_p. $$
\end{example}
\begin{example}
Setting $a_1=2$ we get
$$a_2=-\frac{2(\a-2)}{4\a-3}.$$
Setting now $\a=\frac{p^k+3}{4}$ if $p$ is congruent to 1 modulo 4, or $\a=\frac{p^{2k}+3}{4}$ if $p$ is congruent to $-1$ modulo 4 gives an infinite family of convergent examples of the form
\begin{align*}
    \sqrt{\frac{p^{2k}+6p^k+25}{16}}&=\left [\frac{p^k+3}{4},\overline{2,-\frac{p^k-5}{2p^k},2}\right ]&&\text{if $p\equiv 1\pmod 4$},\\
        \sqrt{\frac{p^{4k}+6p^{2k}+25}{16}}&=\left [\frac{p^{2k}+3}{4},\overline{2,-\frac{p^{2k}-5}{2p^{2k}},2}\right ]&&\text{if $p\equiv -1\pmod 4$}.
\end{align*}
\end{example}

\section*{Acknowledgements}
All the authors are members of the INdAM group GNSAGA.

The authors are grateful to the referees for their careful reading of the manuscript and for their valuable comments and suggestions, which helped improve the exposition of the paper.

\section*{Declarations}

\subsection*{Funding}
The authors received no specific funding for this work.

\subsection*{Conflict of interest}
The authors declare that they have no conflict of interest.

\subsection*{Data availability}
Data sharing is not applicable to this article as no datasets were generated or analysed during the current study.

\printbibliography

@article{Bedocchi1988,
    author  = "Bedocchi, Edmondo",
    title   = "A note on p-adic continued fractions",
    year    = "1988",
    journal = "Annali di Matematica Pura ed Applicata",
    pages   = "197--207"
}

@article{Bedocchi1990,
    author  = "Bedocchi, Edmondo",
    title   = "Sur le developpement de $\sqrt{m}$ en fraction continue p-adique",
    year    = "1990",
    journal = "Manuscripta Mathematica",
    volume  = "67",
    pages   = "187--195"
}

@article{Browkin1978,
    author  = "Browkin, J.",
    title   = "Continued fractions in local fields I",
    year    = "1978",
    journal = "Demonstratio Mathematica",
    volume  = "11",
    pages   = "67--82."
}

@article {BrockElkies2021,
    AUTHOR = {Brock, Bradley W. and Elkies, Noam D. and Jordan, Bruce W.},
     TITLE = {Periodic continued fractions over {$S$}-integers in number
              fields and {S}kolem's {$p$}-adic method},
   JOURNAL = {Acta Arith.},
  FJOURNAL = {Acta Arithmetica},
    VOLUME = {197},
      YEAR = {2021},
    NUMBER = {4},
     PAGES = {379--420},
}

@article{Browkin2000,
    author  = "Browkin, J.",
    title   = "Continued fractions in local fields II",
    year    = "2000",
    journal = "Mathematics of Computation",
    volume  = "70",
    pages   = "1281--1292"
}

@article {CapuanoMurruTerracini2022,
    AUTHOR = {Capuano, Laura and Murru, Nadir and Terracini, Lea},
     TITLE = {On the finiteness of  {$\mathfrak P$}-{adic} continued fractions for number fields},
   JOURNAL = {Bull. Soc. Math. France},
  FJOURNAL = {Bulletin de la Soci\'{e}t\'{e} Math\'{e}matique de France},
    VOLUME = {150},
      YEAR = {2022},
    NUMBER = {4},
     PAGES = {743--772},
}

@article{Bun57,
TITLE = {Sur les diviseurs num\'eriques invariables des fonctions rationnelles enti\`eres},
AUTHOR = {Bounyakowsky, Victor Y.},
noTITLE = {M\'emoires de l'Acad\'emie Im\'eriale des {S}ciences de Saint-P\'etersbourg'},
JOURNAL = {Sixi\`eme s\'erie Sciences Math\'ematiques, Physiques et Naturelles},
VOLUME = {Tome{VIII}, Premi\`ere partie {T}ome {VI} },
YEAR = {1857},
PAGES = {305 -- 329},
}

@article{CapuanoMurruTerracini2023,
      title={On periodicity of $p$-adic Browkin continued fractions}, 
      author={Laura Capuano and Nadir Murru and Lea Terracini},
      year={2023},
      journal = {Math. Z.},
      volume = {35},
      number = {7},
      pages = {1--24},
}

@article{CapuanoVenezianoZannier2019,
    author  = "Capuano, L. and Veneziano, F. and Zannier, U.",
    title   = "An effective criterion for periodicity of $ \ell$-adic continued fractions",
    year    = "2019",
    journal = "Mathematics of Computation",
    volume  = "88",
    number  = "",
    pages   = "1851--1882"
}

@article{BarberoCerrutiMurru2021,
  author  = {Barbero, Stefano and Cerruti, Umberto and Murru, Nadir},
  title   = {Periodic representations for quadratic irrationals in the field of {$p$}-adic numbers},
  journal = {Mathematics of Computation},
  volume  = {90},
  number  = {331},
  year    = {2021},
  pages   = {2267--2280},
  mrnumber = {4280301}
}

@article{Mahler1940,
    author  = "Mahler, K.",
    title   = "On a geometrical representation of p-adic numbers",
    year    = "1940",
    journal = "Ann. of Math.",
    volume  = "41",
    number  = "",
    pages   = "8--56"
}

@article {MurruRomeo2023,
    AUTHOR = {Murru, N. and Romeo, G. and Santilli, G.},
     TITLE = {On the periodicity of an algorithm for $p$-adic continued fractions},
   JOURNAL = {Ann. Math. Pur. Appl.},
    VOLUME = {206},
      YEAR = {2023},
    NUMBER = {6},
     PAGES = {2971-2984},
}

@article{Ooto02017,
    author  = "Ooto, T.",
    title   = "Transcendental $p$--adic continued fractions",
    year    = "2017",
    journal = " Math. Z.",
    volume  = "287",
    number  = "",
    pages   = "1053--1064"
}

@article {Petho1982,
    AUTHOR = {Peth\H{o}, Attila},
     TITLE = {Perfect powers in second order linear recurrences},
   JOURNAL = {J. Number Theory},
  FJOURNAL = {Journal of Number Theory},
    VOLUME = {15},
      YEAR = {1982},
    NUMBER = {1},
     PAGES = {5--13},
}

@article {Ruban1970,
    AUTHOR = {Ruban, A. A.},
     TITLE = {Certain metric properties of the {$p$}-adic numbers},
   JOURNAL = {Sibirsk. Mat. \v{Z}.},
  FJOURNAL = {Akademija Nauk SSSR. Sibirskoe Otdelenie. Sibirski\u{\i}
              Matemati\v{c}eski\u{\i} \v{Z}urnal},
    VOLUME = {11},
      YEAR = {1970},
     PAGES = {222--227}}

@incollection {Schneider1970,
    AUTHOR = {Schneider, Th.},
     TITLE = {\"{U}ber {$p$}-adische {K}ettenbr\"{u}che},
 BOOKTITLE = {Symposia {M}athematica, {V}ol. {IV} ({INDAM}, {R}ome,
              1968/69)},
     PAGES = {181--189},
 PUBLISHER = {Academic Press, London},
      YEAR = {1970}}

@book {Nagell1951,
    AUTHOR = {Nagell, Trygve},
     TITLE = {Introduction to {N}umber {T}heory},
 PUBLISHER = {John Wiley and Sons, 
    Stockholm},
      YEAR = {1951},
}

@incollection{lagrange1770,
  author    = {Lagrange, Joseph-Louis},
  title     = {Additions au mémoire sur la résolution des équations nu\-mé\-riques},
  booktitle = {Oeuvres complètes, tome 2},
  publisher = {Gauthier-Villars (Paris)},
  volume={24},
  year      = {1770}, 
  pages     = {581--652},
  note      = {Reprinted from Mémoires de l’Académie royale des Sciences et Belles-lettres de Berlin, Vol. 24 (1770)},
}

@article {MR2215137,
    AUTHOR = {Bugeaud, Yann and Mignotte, Maurice and Siksek, Samir},
     TITLE = {Classical and modular approaches to exponential {D}iophantine
              equations. {I}. {F}ibonacci and {L}ucas perfect powers},
   JOURNAL = {Ann. of Math. (2)},
  FJOURNAL = {Annals of Mathematics. Second Series},
    VOLUME = {163},
      YEAR = {2006},
    NUMBER = {3},
     PAGES = {969--1018},
}

@book{hindrySilverman2000,
  author    = {Marc Hindry and Joseph H. Silverman},
  title     = {Diophantine Geometry: An Introduction},
  series    = {Graduate Texts in Mathematics},
  publisher = {Springer},
  address   = {New York, NY},
  year      = {2000},
  edition   = {1},
  noisbn      = {978-0-387-98975-4},
}

@article{Romeo2024,
  author  = {Giuliano Romeo},
  title   = {Continued fractions in the field of $p$-adic numbers},
  journal = {Bulletin of the American Mathematical Society},
  volume  = {61},
  year    = {2024},
  pages   = {343--371},
}

@article{Murru2023,
  author  = {Murru, Nadir and Romeo, Giuliano and Santilli, Giordano},
  title   = {Convergence conditions for p-adic continued fractions},
  journal = {Research in Number Theory},
  year    = {2023},
  volume  = {9},
  number  = {3},
  pages   = {66},
}
\end{document}